\documentclass[12pt]{amsart}
\usepackage{txfonts}
\usepackage{amscd,amssymb,mathabx,subfigure}
\usepackage[all]{xy}
\usepackage{graphicx,calrsfs}
\usepackage[colorlinks,plainpages,backref,urlcolor=blue,breaklinks=true]{hyperref}
\usepackage{breakurl}
\usepackage{bbm}
\usepackage{tikz}
\usetikzlibrary{calc}
\usepackage{mdwlist}
\usepackage{enumitem}
\usepackage{verbatim}

\newtheorem{theorem}{Theorem}[section]
\newtheorem{corollary}[theorem]{Corollary}
\newtheorem{lemma}[theorem]{Lemma}

\theoremstyle{definition}
\newtheorem{definition}[theorem]{Definition}
\newtheorem{example}[theorem]{Example}
\newtheorem{remark}[theorem]{Remark}

\newtheorem*{ack}{Acknowledgments}

\newcommand{\Z}{\mathbb{Z}}
\newcommand{\Q}{\mathbb{Q}}

\newcommand{\C}{\mathbb{C}}

\renewcommand{\L}{\mathbb{L}}

\newcommand{\PP}{\mathbb{P}}
\newcommand{\CP}{\mathbb{CP}}

\DeclareMathAlphabet{\pazocal}{OMS}{zplm}{m}{n}

\newcommand{\YY}{{\pazocal Y}}

\newcommand{\RR}{{\mathcal R}}
\newcommand{\VV}{{\mathcal V}}

\newcommand{\F}{{\mathcal{F}}}

\newcommand{\cE}{{\mathcal{E}}}

\newcommand{\X}{{\mathcal{X}}}

\newcommand{\g}{{\mathfrak{g}}}
\newcommand{\h}{{\mathfrak{h}}}
\newcommand{\gl}{{\mathfrak{gl}}}
\renewcommand{\sl}{{\mathfrak{sl}}}

\newcommand{\sol}{{\mathfrak{sol}}}

\DeclareMathOperator{\im}{im}

\DeclareMathOperator{\id}{id}
\DeclareMathOperator{\ab}{{ab}}

\DeclareMathOperator{\GL}{GL}
\DeclareMathOperator{\SL}{SL}
\DeclareMathOperator{\PSL}{PSL}
\DeclareMathOperator{\OS}{OS}

\DeclareMathOperator{\Hom}{{Hom}}

\DeclareMathOperator{\ad}{ad}

\DeclareMathOperator{\ideal}{ideal}

\DeclareMathOperator{\orb}{orb}
\DeclareMathOperator{\abf}{abf}

\DeclareMathOperator{\Sol}{Sol}
\DeclareMathOperator{\Zero}{\mathbf{V}}

\newcommand{\surj}{\twoheadrightarrow}
\newcommand{\inj}{\hookrightarrow}

\newcommand\isom{\xrightarrow{
 \,\smash{\raisebox{-0.4ex}{\ensuremath{\scriptstyle\simeq}}}\,}}

\def\dot{\mathchar"013A}  
\newcommand{\hdot}{{\raise2pt\hbox to0.35em{\Huge $\dot$}}} 
 
\newcommand{\bwedge}{\mbox{\small $\bigwedge$}}

\newcommand{\oM}{\overline{M}}
\newcommand{\oC}{\overline{C}}

\newcommand{\Lie}{\textup{\texttt{Lie}}}

\newcommand{\dga}{\large{\textsc{cdga}}}

\definecolor{dkgreen}{RGB}{0,100,0}
\definecolor{dkbrown}{RGB}{139,69,19}

\numberwithin{equation}{section}

\begin{document}

\title[Rank $2$ jump loci of quasi-projective manifolds]{%
Rank two topological and infinitesimal embedded \\ jump loci of
quasi-projective manifolds}

\author[Stefan Papadima]{Stefan Papadima$^1$}
\address{Simion Stoilow Institute of Mathematics, 
P.O. Box 1-764,
RO-014700 Bucharest, Romania}
\email{\href{mailto:Stefan.Papadima@imar.ro}{Stefan.Papadima@imar.ro}}
\thanks{$^1$Work partially supported by the Romanian Ministry 
of Research and Innovation, CNCS--UEFISCDI, grant
PN-III-P4-ID-PCE-2016-0030, within PNCDI III}

\author[Alexander~I.~Suciu]{Alexander~I.~Suciu$^2$}
\address{Department of Mathematics,
Northeastern University,
Boston, MA 02115, USA}
\email{\href{mailto:a.suciu@northeastern.edu}{a.suciu@northeastern.edu}}
\urladdr{\href{http://web.northeastern.edu/suciu/}%
{web.northeastern.edu/suciu/}}
\thanks{$^2$Partially supported by the Simons Foundation 
collaboration grant for mathematicians 354156}



\begin{abstract}
We study the germs at the origin of $G$-representation varieties  
and the degree $1$ cohomology jump loci of fundamental groups of  
quasi-projective manifolds.  Using the Morgan--Dupont model 
associated to a convenient compactification of such a manifold, 
we relate these germs to those of their infinitesimal  counterparts, 
defined in terms of flat connections on those models. When the 
linear algebraic group $G$ is either $\SL_2(\C)$ or its standard Borel 
subgroup and the depth of the jump locus is $1$, this dictionary works 
perfectly, allowing us to describe in this way explicit irreducible 
decompositions for the germs of these embedded jump loci. On the 
other hand, if either $G=\SL_n(\C)$ for some $n\ge 3$, or the depth 
is greater than $1$, then certain natural inclusions of germs are strict. 
\end{abstract}

\maketitle
\setcounter{tocdepth}{1}
\tableofcontents
\newpage
\section{Introduction and statement of results}
\label{sect:intro}

\subsection{Representation varieties and cohomology jump loci}
\label{subsec:intro1}

Let $X$ be a path-connected space having the homotopy 
type of a finite CW-complex, and let $G$ be a complex, 
linear algebraic group. The representation 
variety $\Hom(\pi_1(X),G)$ is the parameter space for locally 
constant sheaves on $X$ whose monodromy factors through  $G$. 
The {\em characteristic varieties}\/ of $X$ with respect to a 
rational representation $\iota\colon G\to \GL(V)$ are the jump loci
\begin{equation}
\label{eq:jumps}
\VV^i_r(X, \iota)=
\{ \rho \in \Hom(\pi_1(X),G) \mid \dim H^i(X,  V_{\iota \circ \rho}) \ge r \},
\end{equation}
defined for each degree $i\ge 0$ and depth $r\ge 0$.  These loci,  
which record the variation of twisted cohomology inside the parameter space, 
encode subtle information about the topology of $X$ and its covering spaces.  
We focus here mainly on the degree $1$ jump loci, which depend only on the 
fundamental group $\pi_1(X)$ and the representation $\iota$.

We work throughout over $\C$, unless otherwise mentioned. 
For an affine variety $\X$, we denote by $\X_{(x)}$ the analytic germ of $\X$ 
at a point $x\in \X$.  Affine varieties and analytic germs are always reduced.

The study of analytic germs of embedded cohomology jump loci is a basic problem in 
deformation theory with homological constraints.  Building on the foundational 
work of Goldman and Millson \cite{GM}, it was shown in \cite{DP-ccm} that 
the germs at the origin $1\in \Hom(\pi_1(X),G)$ of those loci are 
isomorphic to the germs at the origin of embedded infinitesimal jump loci
of a commutative differential graded algebra (for short, $\dga$) that is a 
finite model for the space $X$.    In \cite{BW}, Budur and Wang 
extended this result away from the origin, by developing a theory of 
differential graded Lie algebra modules which control the corresponding 
deformation problem. 

The case when $G=\SL_2(\C)$  has received a lot 
of attention in the literature. For instance, results from \cite{PS-17} reveal 
an unexpected connection between $\SL_2(\C)$ representation varieties 
and the monodromy action on the homology of Milnor fibers 
of central hyperplane arrangements: for a line arrangement in $\CP^2$, 
combinatorial information (namely the nonexistence of points of intersection 
with multiplicity properly divisible by $3$) implies the fact that all  roots of 
the characteristic polynomial of the monodromy action on the first homology
of the Milnor fiber of order a power of $3$ have multiplicity at most $2$ (a delicate 
topological property). The proof from \cite[\S7.3]{PS-17} uses a construction 
related to the irreducible decomposition of the analytic germ 
$\Hom(\pi_1(X), \SL_2(\C))_{(1)}$, where $X$ is the arrangement complement. 
On the other hand, the universality theorem of Kapovich and Millson \cite{KM} 
shows that rank $2$ representation varieties may have arbitrarily bad singularities 
away from $1$. This lead us to focus on germs at the origin of such varieties, 
and look for explicit descriptions via infinitesimal $\dga$ methods.

\subsection{Flat connections and resonance varieties}
\label{subsec:intro2}

The infinitesimal analogue of the $G$-represen\-tation variety 
around the origin is the set $\F(A,\g)$ of $\g$-valued flat connections on 
a commutative, differential, positively-graded $\C$-algebra $A$, where $\g$ is the 
Lie algebra of the Lie group $G$.  This set  consists of all elements 
$\omega \in A^1\otimes \g$ which satisfy the Maurer--Cartan equation, 
$d\omega +\frac{1}{2}[\omega,\omega]=0$.  If 
$A^1$ is finite dimensional, then the set of flat connections is a Zariski-closed
subset of the affine space $A^1\otimes \g$.  Furthermore, $\F(A,\g)$ contains 
the closed subvariety $\F^1(A,\g)$ consisting of all tensors of the form 
$\eta \otimes g$ with $d \eta=0$. 

To define the infinitesimal counterpart of the jump loci \eqref{eq:jumps}, 
let $\theta\colon \g\to \gl(V)$ be a finite-dimensional representation.  
For each $\g$-valued flat connection $\omega$, there is an associated 
covariant derivative, $d_{\omega}\colon A^{\hdot}\otimes V\to 
A^{\hdot+1}\otimes V$, given by $d_{\omega}=d\otimes \id_V +\ad_{\omega}$,   
and satisfying $d_{\omega}^2=0$. 
The {\em resonance varieties}\/ of $A$ with respect to $\theta$ are the sets 
\begin{equation}
\label{eq:resonant}
\RR^i_r(A,\theta)=
\{ \omega \in \F(A,\g) \mid \dim H^i(A\otimes V, d_{\omega})\ge r \}. 
\end{equation} 

If $A$ is finite-dimensional, then these sets are Zariski-closed in $\F(A,\g)$. 
Furthermore, if $H^i(A)\ne 0$, then $\RR^i_1(A,\theta)$ 
contains the closed subvariety $\Pi(A,\theta)$  
consisting of all elements $\eta \otimes g\in \F^1(A,\g)$ 
with $\det \theta(g)=0$. 

\subsection{Quasi-K\"{a}hler manifolds and admissible maps}
\label{subsec:intro3}
We now turn our attention to a class of spaces for which the characteristic 
varieties are constrained by some powerful structural results. 
Let $M$ be a quasi-K\"{a}hler manifold, that is, the complement 
of a normal crossing divisor $D$ in a compact, connected  K\"{a}hler 
manifold $\oM$.  A map $f\colon M \to C$ from such a manifold to a 
smooth complex curve $C$ is said to be {\em admissible}\/ if $f$ is 
holomorphic and surjective, and $f$ admits a holomorphic, surjective 
extension between suitable compactifications, $\bar{f}\colon \oM \to \oC$, 
such that all the fibers of $\bar{f}$ are connected. 

As shown by Arapura in \cite{Ar}, there exists a finite set $\cE(M)$ of equivalence 
classes of `admissible' maps from $M$ to smooth curves of negative Euler 
characteristic, up to reparametrization in the target.  
For each such map 
$f\colon M\to M_f$, we denote by $f_{\sharp}\colon \pi\to \pi_f$ the 
induced homomorphism on fundamental groups; the admissibility 
condition insures that $f_{\sharp}$ is surjective. 
Let $\abf \colon \pi \surj \pi_{\abf}$ be the projection of the group 
$\pi$ onto its maximal torsion-free abelian quotient. 
We will denote by $f_0 \colon M \to K(\pi_{\abf}, 1)$ 
the corresponding classifying map, which is determined 
up to homotopy by the property that $(f_0)_{\sharp}= \abf$.   
Furthermore, we will write $E(M)=\cE(M)\cup \{ f_0\}$.

Let $G$ be a complex linear algebraic group, 
let $\iota\colon G\to \GL(V)$ be a rational  representation, and let 
$\theta \colon \g \to \gl(V)$ be its tangential representation.  
For all $r\ge 0$, we have inclusions
\begin{equation}
\label{eq:vincldeg1-intro}
\VV^1_r(\pi,\iota) \supseteq \bigcup_{f\in E(M)} f_{\sharp}^{*} \VV^1_r (\pi_f,\iota),
\end{equation}
where $f_{\sharp}^{*}\colon \Hom (\pi_f,G) \to \Hom (\pi,G)$ denotes the induced 
morphism on representation varieties. For $r=0$ and $1$, the inclusions from 
\eqref{eq:vincldeg1-intro} are equivalent to the two inclusions
\begin{align}
\label{eq:repincl-intro}
\Hom(\pi,G) &\supseteq  \abf^{*} \Hom (\pi_{\abf},G) \cup 
\bigcup_{f\in \cE(M)} f_{\sharp}^{*} \Hom (\pi_f,G),
\\
\label{eq:vincl-intro}
\VV^1_1(\pi,\iota) &\supseteq \abf^{*} \VV^1_1 (\pi_{\abf},\iota) \cup 
\bigcup_{f\in \cE(M)} f_{\sharp}^{*} \Hom (\pi_f,G).
\end{align}

The case when $b_1(M)=0$ is trivial, since $\Hom(\pi,G)_{(1)}=\{ 1\}$, while 
both $\VV^1_1(\pi,\iota)_{(1)}$ and $\cE(M)$ are empty in that situation. 
So, it is harmless to assume that $b_1(M)>0$.

In the rank $1$ case, i.e., the case when $G=\C^{\times}$ 
and $\iota$ identifies $\C^{\times}$ with $\GL_1(\C)$, equality near 
$1$ in \eqref{eq:vincl-intro} holds, by a deep result of 
Arapura \cite{Ar} on the structure of $\VV^1_1 (\pi,\iota)_{(1)}$.  
In particular, every nontrivial character $\rho\in \Hom( \pi, \C^{\times})$ 
sufficiently close to the trivial character and such that 
$H^1(\pi, \C_{\rho})\ne 0$ must belong to 
$f_{\sharp}^{*} \Hom (\pi_f,\C^{\times})$, for some $f\in \cE(M)$. 
For a more general treatment of factorization results of this 
nature we refer to the book by  Zuo \cite{Zu} and to 
the recent work of Campana, Claudon, and 
Eyssidieux \cite{CCE14, CCE15}. 

\subsection{Quasi-projective manifolds and transversality}
\label{subsec:intro4}
We specialize now to the case when $M$ is a connected, smooth 
quasi-projective variety (for short, a quasi-projective manifold). 
Let $(\oM,D)$ be a convenient compactification of $M$, where 
$\oM$ is now a projective manifold, and $D$ is a union of smooth 
hypersurfaces, intersecting locally like hyperplanes.  
Work of Morgan \cite{Mo}, as recently sharpened  
by Dupont in \cite{Du}, associates to these data a bigraded, 
rationally defined $\dga$, $A=\OS(\overline{M},D)$, 
called the Orlik--Solomon model of $M$.  This $\dga$ 
is a finite model of $M$, i.e., it is connected ($A^0=\C$),   
finite-dimensional as a $\C$-vector space,  
and weakly equivalent to the de Rham algebra of $M$. 
Furthermore, $A$  is functorial with respect to regular 
morphisms of pairs $(\oM,D)$ as above. 

For an admissible map $f\colon M\to M_f$, 
we will denote by $\Phi_f\colon A_f \to A$ the induced $\dga$ map 
between the respective Orlik--Solomon models. Let $ \Phi_f^{*}\colon  \F(A_f,\g) \to \F(A,\g)$
be the morphism induced by $\Phi_f$ between the respective varieties of flat connections.
Assuming as before that $b_1(M)>0$, we obtain the following infinitesimal 
counterparts of inclusions \eqref{eq:repincl-intro}--\eqref{eq:vincl-intro}:
\begin{align}
\label{eq:flatincl-intro}
\F(A,\g) &\supseteq \F^1(A,\g)\cup 
\bigcup_{f\in \mathcal{E}(M)}  \Phi_f^{*} \F(A_f,\g),
\\
\label{eq:rincl-intro}
\RR^1_1(A,\theta) &\supseteq  \Pi(A,\theta)\cup 
\bigcup_{f\in \mathcal{E}(M)}  \Phi_f^{*} \F(A_f,\g).
\end{align}

This brings us to our first result, which can be viewed as a `transversality' theorem 
for the subvarieties which appear on the right-hand side of inclusions 
\eqref{eq:repincl-intro}--\eqref{eq:rincl-intro}.  
This result  summarizes Theorems \ref{thm:nonabtrans} and \ref{thm:toptrans}, 
and will be proved in Sections \ref{subsec:trans} and \ref{subsec:trans-bis}.

\begin{theorem}
\label{thm:tran}
Let $M$ be a quasi-K\"{a}hler manifold, and 
let $f,g\in \cE(M)$ be two distinct admissible maps.  
\begin{enumerate}
\item \label{fg1}
If $M$ is a quasi-projective manifold, then 
\[
\Phi_f^{*} \F(A_f,\g) \cap \Phi_g^{*} \F(A_g,\g)= \{ 0\}.
\]
\item \label{fg2}
If $M$ is either a compact, connected K\"{a}hler manifold or the complement 
of a central complex hyperplane arrangement, then
\[
f_{\sharp}^{*} \Hom (\pi_f,G)_{(1)} \cap g_{\sharp}^{*} \Hom (\pi_g,G)_{(1)} = \{ 1\}.
\]
\end{enumerate}
\end{theorem}

In the rank $1$ case, part \eqref{fg2} of the theorem also follows from results 
in \cite{DPS-duke}.  Moreover, if $M$ is an arrangement complement, an 
equivalent statement can be found in \cite{LY}, again only in the rank $1$ case. 
The novelty here is that a completely analogous statement holds for arbitrary 
complex linear algebraic groups $G$.

\subsection{Topological versus infinitesimal factorizations}
\label{subsec:intro5}

Our main goal in this paper is to analyze the decomposition 
into irreducible components of the germs at $1$ of the embedded jump loci 
$(\Hom(\pi,G), \VV^1_1(\pi, \iota))$ and the germs at $0$ of their infinitesimal 
analogues, $(\F(A,\g), \RR^1_1(A,\theta))$, in the case when $\pi$ is the 
fundamental group of a quasi-projective manifold. 

A key step in this direction is the next theorem, which 
establishes a very strong connection between equalities in 
\eqref{eq:repincl-intro}--\eqref{eq:rincl-intro}, and opens the way for using infinitesimal
computations to derive factorization results near $1$.

\begin{theorem}
\label{thm:main}
Let $M$ be quasi-projective manifold  with $b_1(M)>0$.  
For an arbitrary rational representation of $G=\SL_2(\C)$ 
or its standard Borel subgroup $\Sol_2(\C)$, 
the following statements are equivalent.
\begin{enumerate}
\item \label{m1}
The inclusion \eqref{eq:repincl-intro} becomes an equality near $1$.
\item \label{m2}
Both \eqref{eq:repincl-intro} and \eqref{eq:vincl-intro} become 
equalities near $1$.
\item \label{m3}
The inclusion \eqref{eq:flatincl-intro} is an equality, for some 
convenient compactification of $M$.
\item \label{m4}
Both \eqref{eq:flatincl-intro} and \eqref{eq:rincl-intro} are equalities, 
for any convenient compactification of $M$.
\end{enumerate}
\end{theorem}

This theorem, which will be proved in \S\ref{subsec:pfmain}, 
provides a topological interpretation for Question 8.4 from \cite{PS-15}, 
which asks whether statement \eqref{m4} from above always holds.  

\subsection{Irreducible decompositions for germs of embedded jump loci}
\label{subsec:intro6}

The next theorem, which will be proved in \S\ref{subsec:pfirr}, 
is our main result regarding the irreducible decomposition around 
the origin of the rank $2$ topological and infinitesimal embedded jump loci 
of quasi-projective manifolds. 

\begin{theorem}
\label{thm:irrtran}
With notation as above, suppose the equivalent properties 
from Theorem \ref{thm:main} are satisfied.  
\begin{enumerate}
\item \label{t2}
If $b_1(M_f)\ne b_1(M)$ for all $f\in \cE(M)$, then we have the following 
decompositions into irreducible components of analytic germs:
\begin{align}
\label{eq:repincl-irr}
\Hom (\pi,G)_{(1)} &= \abf^{*} \Hom (\pi_{\abf},G)_{(1)} \cup 
\bigcup_{f\in \cE(M)} f_{\sharp}^{*} \Hom (\pi_f,G)_{(1)}, 
\\
\label{eq:vincl-irr}
\VV^1_1 (\pi,\iota)_{(1)} &= \abf^{*} \VV^1_1 (\pi_{\abf},\iota)_{(1)} \cup 
\bigcup_{f\in \cE(M)} f_{\sharp}^{*} \Hom (\pi_f,G)_{(1)},
\\
\label{eq:flatincl-irr}
\F(A,\g)_{(0)} &= \F^1(A,\g)_{(0)} \cup \bigcup_{f\in \mathcal{E}(M)}  
\Phi_f^{*} \F(A_f,\g)_{(0)},
\\
\label{eq:rincl-irr}
\RR^1_1(A,\theta)_{(0)} &=\Pi(A,\theta)_{(0)} \cup 
\bigcup_{f\in \mathcal{E}(M)} \Phi_f^{*} \F(A_f,\g)_{(0)}.
\end{align}

\item  \label{t3}
If $b_1(M_f)= b_1(M)$ for some $f\in \cE(M)$, then we have the 
following equalities of irreducible germs:
\begin{align}
\label{eq:repincl-spec}
\Hom (\pi,G)_{(1)} &= f_{\sharp}^{*} \Hom (\pi_f,G)_{(1)}, 
\\
\label{eq:vincl-spec}
\VV^1_1 (\pi,\iota)_{(1)} &= f_{\sharp}^{*} \Hom (\pi_f,G)_{(1)},
\\
\label{eq:flatincl-spec}
\F(A,\g)_{(0)} &=  \Phi_f^{*} \F(A_f,\g)_{(0)},
\\
\label{eq:rincl-spec}
\RR^1_1(A,\theta)_{(0)} &=\Phi_f^{*} \F(A_f,\g)_{(0)}.
\end{align}

\item  \label{t1}
For any two distinct admissible maps $f,g\in \cE(M)$, 
\[
f_{\sharp}^{*} \Hom (\pi_f,G)_{(1)} \cap g_{\sharp}^{*} \Hom (\pi_g,G)_{(1)} = \{ 1\}.
\]
\end{enumerate}
\end{theorem}

Under our assumptions, 
this theorem gives a local, more precise and simple, classification 
for representations of $\pi$ into $G$, when compared to the similar 
global, more sophisticated classification obtained by Corlette--Simpson \cite{CS} 
and Loray--Pereira--Touzet \cite{LPT} in the case when $G=\SL_2(\C)$,   
see Remark \ref{rem:shimura}. Furthermore, as explained in Remark \ref{rk:known}, 
all irreducible components appearing in Theorem \ref{thm:irrtran} are known, 
for an arbitrary quasi-projective manifold with $b_1(M)>0$, and for 
an arbitrary rational representation of $G=\SL_2(\C)$ or $\Sol_2(\C)$, 

\subsection{Applying the decomposition results}
\label{subsec:intro7}

By now, the reader may wonder whether our structural results on 
the irreducible decompositions of germs of embedded jump loci 
apply in any meaningful way.  The next theorem, which will be 
proved in Section \ref{subsec:top}, seeks to dispel such possible 
doubts, by providing a rich supply of quasi-projective manifolds for which 
both \eqref{eq:repincl-intro} and \eqref{eq:vincl-intro} hold as equalities 
near $1$, in the rank two case.

\begin{theorem}
\label{thm:loci}
Suppose $M$ is quasi-projective manifold satisfying one 
of the following hypotheses. 
\begin{enumerate}
\item \label{ai1}
$M$ is projective.
\item \label{ai2}
The Deligne weight filtration has the property that $W_1H^1(M)=0$.
\item \label{ai3}
$M$ is the partial configuration space of a projective curve associated to 
an arbitrary finite simple graph.
\item \label{ai4}
$\RR^1_1(H^{\hdot}(M), d=0)= \{ 0\}$.
\item \label{ai5}
$M=S \setminus \{ 0\}$, where $S$ is a quasi-homogeneous affine surface 
having a normal, isolated singularity at $0$. 
\end{enumerate}
Then, for $G=\SL_2(\C)$ or $\Sol_2(\C)$, the equivalent properties 
from Theorem \ref{thm:main} are satisfied, and thus, the conclusions 
of Theorem \ref{thm:irrtran} hold. 
\end{theorem}

It would be interesting to decide whether the conclusions 
of Theorem \ref{thm:irrtran} hold for arbitrary quasi-projective 
manifolds, not just the ones from the above list. 

\subsection{Beyond depth $1$ or rank $2$}
\label{subsec:intro8}

As shown in the next result (which will be proved in Theorem \ref{thm:rank3}), 
the higher-rank analogue of the local equality \eqref{eq:repincl-irr} fails, even 
for some very simple, weighted-homogeneous quasi-projective surfaces. 

\begin{theorem}
\label{thm:k3}
Let $M=S \setminus \{ 0\}$, where $S$ is a quasi-homogeneous affine surface 
having a normal, isolated singularity at $0$.  If 
$b_1(M)>0$ and $G=\SL_n(\C)$ with $n\ge 3$, 
then inclusion \eqref{eq:repincl-intro} is strict near $1$. 
\end{theorem}

Finally, as shown in the next result (which will be proved in Theorem \ref{thm:depth2}), 
the higher-depth analogue of Theorem \ref{thm:irrtran} also may fail, even in rank $2$.

\begin{theorem}
\label{thm:dep2}
Let $M$ be a connected, compact K\"{a}hler manifold, or the complement 
of a central complex hyperplane arrangement, and suppose  
there exists an admissible map $f\colon M\to M_f$ such that $b_1(M_f)<b_1(M)$.  
If  $\iota\colon G\to \GL(V)$ is  
a rational representation of $G=\SL_2(\C)$ or $\Sol_2(\C)$, having a non-zero 
fixed vector $v\in V^G$, there is then an integer $r>1$ such that inclusion 
\eqref{eq:vincldeg1-intro} is strict near $1$.
\end{theorem}

Concrete instances where this theorem applies are given in 
Examples \ref{ex:arrk} and \ref{ex:cpt}. On the other hand, when 
$M$ is a connected, compact K\"{a}hler manifold or the complement 
of a central complex hyperplane arrangement and $\iota\colon G\to \GL(V)$ is  
an arbitrary rational representation of $\SL_2(\C)$ or $\Sol_2(\C)$, then, 
as shown in \cite{PS-16}, the local equalities \eqref{eq:repincl-irr} 
and \eqref{eq:vincl-irr} hold.  

\section{Local analytic germs}
\label{sect:local}

\subsection{Irreducible decompositions}
\label{subsec:irred decomp}

This section contains the necessary preparatory material pertaining to 
irreducible decompositions of complex affine varieties and analytic germs.    
We start with a lemma which will be used repeatedly in the sequel. 

\begin{lemma}
\label{lem:irrtrick}
Let $X$ be an analytic germ, and assume that 
\[
\bigcup_{i\in I} Y_i \subseteq X = \bigcup_{i\in I} X_i,
\] 
where the indexing set $I$ is finite, all germs $X_i, Y_i$ are irreducible, 
$\dim X_i=\dim Y_i$ for all $i$,
and $Y_i \not\subseteq Y_j$ for $i\ne j$. Then:
\begin{enumerate}
\item \label{xx1}  There is a bijection $b\colon I\to I$
such that $Y_i=X_{b(i)}$, for all $i$;
\item \label{xx2} $X = \bigcup_{i\in I} Y_i$.
\end{enumerate}
\end{lemma}

\begin{proof}
Let $J\subseteq I$ be a subset. To prove part \eqref{xx1}, 
we will construct by induction on the cardinality of $J$ an injection 
$b\colon J\inj I$ with the property that $Y_i=X_{b(i)}$, for all $i\in J$, starting 
with $J= \O$. It will be useful to consider the dimension partition, 
$I=\coprod I_d$, where 
\begin{equation}
\label{eq:idxy}
I_d= \{ i\in I \mid \dim X_i=\dim Y_i=d\}.
\end{equation}
Clearly, the injection $b$ must respect the partition blocks.

For the induction step, assume $J\neq I$ and pick $i_0\in I\setminus J$ 
such that $d_0=\dim Y_{i_0}$ maximizes $\dim Y_i$ for $i\in I\setminus J$. 
Plainly, $I_d\subseteq J$ and $b(I_d)=I_d$, for $d>d_0$.

Since $Y_{i_0}$ is irreducible, $Y_{i_0}\subseteq X_i$, for some $i\in I$. 
Set $d=\dim X_i\ge d_0$. If $d>d_0$, then $i=b(j)$ for some 
$j\in I_d\subseteq J$, by the previous remark. This implies that $X_i=Y_j$,
by the induction assumption. We infer that $Y_{i_0}\subseteq Y_j$, 
a contradiction. Hence, $d=d_0$, and
therefore $Y_{i_0}= X_i$, by irreducibility.

Set $J'=J\coprod \{ i_0\}$ and extend $b$ to $J'$ by defining $b(i_0)=i$. 
Then clearly $Y_j=X_{b(j)}$, for all $j\in J'$.  To finish the induction, we 
have to check that $i$ cannot be of the form $b(j)$ with $j\in J$. Otherwise, 
$X_i=Y_j$, by the induction hypothesis. This implies that $Y_{i_0} =Y_j$, 
a contradiction. 

Finally, the equality from part \eqref{xx2} is a direct consequence 
of part \eqref{xx1}. 
\end{proof}

\subsection{From inclusions to equalities}
\label{subsec:incl}
Next, we delineate conditions under which 
inclusions of affine varieties or analytic germs become equalities. 

\begin{lemma}
\label{lem:p1}
Let $\X$ be an affine variety with the property that all irreducible 
components pass through $x\in \X$.  If $\X'$ is another affine variety 
such that the inclusion $\X\subseteq \X'$ holds near $x$, then the 
inclusion holds globally.
\end{lemma}

\begin{proof}
The argument from the first paragraph of \cite[\S 9.23]{DP-ccm} 
establishes the claim.
\end{proof}

\begin{lemma}
\label{lem:p3}
If $X$ and $Y$ are isomorphic germs, and $X\subseteq Y$, then $X=Y$.
\end{lemma}

\begin{proof}
The inclusion $X\subseteq Y$ induces an epimorphism on coordinate rings, 
which must be an isomorphism, by the Hopfian property of Noetherian rings, 
see e.g. \cite[p.~65]{T}.
\end{proof}

\subsection{Local versus global irreducibility}
\label{subsec:loc global}
Finally, we describe a setting in which global and local irreducibility are equivalent.  
We say that an affine subvariety $\X\subseteq \C^n$ has {\em positive weights}\/ 
if $\X$ is invariant with respect to a $\C^{\times}$-action on $\C^n$ 
with positive weights. 

\begin{lemma}
\label{lem:p2}
If $\X$ has positive weights, then all its irreducible components 
pass through $0$. Moreover, the global irreducibility 
of $\X$ is equivalent to the local irreducibility of the germ $\X_{(0)}$.
\end{lemma}

\begin{proof}
Since the algebraic group $\C^{\times}$ is connected, the action by $\C^{\times}$ 
on $\C^n$ leaves the irreducible components of $\X$ invariant.  Fix such a 
component $\YY$, and let $x\in \YY$.  Then $tx \in \YY$ for all $t \in \C^{\times}$, 
and thus $0=\lim_{t\to 0} tx$ must also belong $\YY$, since the action has 
positive weights. 

To prove the second claim, let us consider the canonical algebra morphisms,
$R\to S\to T$, relating polynomials, convergent series and formal series in 
$n$ variables.  Given an ideal $I \subset R$, denote by $J$ and $K$ the 
ideals generated by $I$ in $S$ and $T$, respectively. When $I$ is the 
defining ideal of an affine subvariety $\X \subseteq \C^n$ having positive 
weights, we know that $I$ is generated by finitely many polynomials 
which are homogeneous 
with respect to the positive weights of the variables. We infer that
an element $f\in R$ (respectively $f\in T$) belongs to $I$ (respectively $K$) 
if and only if all its weighted-homogeneous components $f_i$ are in $I$. 
Hence, $R/I$ canonically embeds into $T/K$.  Consequently, if $T/K$ is 
a reduced ring (or a domain), the ring $R/I$ has the same property. We 
claim that both implications above are actually equivalences. 

Granting this claim, we may finish our proof, as follows. Given an arbitrary 
ideal $J\subset S$, it is well-known that the canonical algebra morphism, 
$S/J \to T/K$ (where $K$ is as above), is injective, cf.~\cite[p.~36]{T}.
It follows from \cite[Cor.~ II.4.2 and Thm.~II.4.5]{T} that $T/K$ is a reduced 
ring (a domain) if and only if the ring $S/J$ has the same property.  Together 
with the above claim, this shows that $R/I$ is a reduced ring (a domain) if 
and only if the ring $S/J$ has the same property. Since $R/I$ and $S/J$ 
are the coordinate rings of $\X$ and $\X_{(0)}$, respectively, we infer that 
$\X$ is irreducible if and only if  $\X_{(0)}$ is irreducible, as asserted.

Going back to the above claim, let us show that if $R/I$ a domain then $T/K$ 
is a domain as well.  (The reduced property can be verified by a similar argument.) 
Otherwise, we may find two formal series with the property that $f\not\equiv 0$ 
(modulo $K$) and $g\not\equiv 0$ (modulo $K$), for which $fg\equiv 0$ (modulo $K$). 
Plainly, we may assume that their weighted initial terms, $f_p \in R$ (respectively, 
$g_q \in R$), do not belong to $I$. But then the initial term of the product, $f_p g_q$, 
must belong to $I$. This contradiction completes our proof.
\end{proof}

\section{Embedded jump loci}
\label{sect:jumps}

\subsection{Representation varieties and characteristic varieties}
\label{subsec:repcv} 

Let $\pi$ be a discrete group, and let $G$ be a $\C$-linear algebraic group. 
The set $\Hom(\pi, G)$ of group homomorphisms from $\pi$ to $G$, called the 
{\em $G$-representation variety}\/ of $\pi$, depends 
bi-functorially on $\pi$ and $G$.  Furthermore, this set comes equipped with 
a natural base point, namely, the trivial representation, $1$. 

Assuming now that $\pi$ is a finitely generated group, 
the set $\Hom(\pi, G)$ acquires a natural structure of affine variety.
Furthermore, every homomorphism $\varphi\colon \pi \to \pi'$ 
induces an algebraic morphism between the corresponding 
representation varieties, $\varphi^{*}\colon \Hom (\pi', G) \to\Hom (\pi, G)$.  

Now let $(X,x_0)$ be a pointed, path-connected space, and let 
$\pi=\pi_1(X,x_0)$ be its fundamental group. Then the representation 
variety $\Hom(\pi, G)$ is the parameter space for finite-dimensional 
local systems on $X$ of type $G$, see e.g.~\cite[Ch.~VI]{Wh}.   
Given a representation $\tau\colon \pi\to \GL(V)$, we let 
$V_\tau$ denote the local system on $X$ associated to $\tau$, 
that is, the left $\pi$-module $V$ defined by $g\cdot v=\tau(g) v$. 
Furthermore, we let $H^{\hdot}(X,V_\tau)$ be the twisted cohomology 
of $X$ with coefficients in this local system, as in \cite[Ch.~VI]{Wh}. 
(If $X$ is semilocally $1$-connected, a classical result of S.~Eilenberg 
identifies twisted homology on $X$ with equivariant homology on the
universal cover of $X$.)

\begin{definition}
\label{def:cv}
The {\em characteristic varieties}\/ of the space $X$ in degree $i\ge 0$ 
and depth $r\ge 0$ with respect to a  representation 
$\iota\colon G \to \GL(V)$ are the sets 
\[
\VV^i_r(X,\iota) =\{ \rho \in \Hom(\pi,G) \mid 
\dim H^i(X, V_{\iota \circ \rho} )\ge r\}.
\]
\end{definition}

For instance, $\VV^0_1(X,\iota)$ consists of those representations 
$\rho\colon \pi\to G$ for which there exists a non-zero vector $v\in V$ 
such that $\iota(\rho(g))(v)=v$, for all $g\in \pi$.  In the rank $1$ case, 
i.e., when $\iota$ is the canonical identification $\C^{\times}\to \GL_1(\C)$, 
we will drop the map $\iota$ from the notation, and simply write $\VV^i_r(X)$.  
For each $i\ge 0$, the sequence $\{\VV^i_r(X,\iota)\}_{r\ge 0}$ is 
a descending filtration of $\Hom(\pi,G) = \VV^i_0(X,\iota)$. 
We will refer to the pairs 
\begin{equation}
\label{eq:homcv}
\left( \Hom(\pi,G) , \VV^i_r(X,\iota) \right)
\end{equation}
as the (global) {\em embedded jump loci}\/ of $X$ with respect to $\iota$. 
Clearly, such pairs depend only on the homotopy type of $X$ and on the 
representation $\iota$.  

Assume now that the space $X$ has the homotopy type of a finite, 
connected CW-complex (in particular, $X$ is path-connected and 
locally simply-connected), and that the map $\iota\colon G \to \GL(V)$ 
is a rational representation.  Then the sets $\VV^i_r(X,\iota)$ are closed 
subvarieties of the representation variety $\Hom(\pi,G)$. The following 
simple example will be useful later on.

\begin{example}
\label{ex:2cw}
Let $X$ be a connected, $2$-dimensional CW-complex, 
and assume $\chi(X)<0$. Then $\VV^1_1(X,\iota)=\Hom(\pi_1(X),G)$, 
for any rational representation $\iota\colon G\to \GL(V)$.  Indeed, 
let $\rho\colon \pi_1(X)\to G$ be a homomorphism. Writing 
$b_i(X,\rho)=\dim H^i(X, V_{\iota \circ \rho} )$, we have that 
\[
b_0(X,\rho) - b_1(X,\rho) + b_2(X,\rho) = \chi(X)\dim(V) <0. 
\]
This forces $b_1(X,\rho)>0$, thereby showing that $\rho\in \VV^1_1(X,\iota)$.
\end{example}

The embedded jump loci enjoy a useful naturality property, which we 
record in the next lemma (see \cite[Cor.~5.8]{PS-16} for a proof, in a 
more general setting).

\begin{lemma}
\label{lem:cvnat}
Let $f\colon X\to X'$ be a pointed map between two 
spaces as above.  Assume that the induced homomorphism 
on fundamental groups, $f_{\sharp}\colon \pi_1(X)\to \pi_1(X')$,  
is surjective. Then the morphism induced by $f_{\sharp}$ 
on representation varieties,
\begin{equation}
\label{eq:ztop}
\xymatrixcolsep{18pt}
\xymatrix{
f_{\sharp}^{*}\colon \Hom (\pi_1(X'), G) \ar[r]& \Hom (\pi_1(X), G)},
\end{equation}
is a closed embedding, which restricts to embeddings 
$\VV^i_r(X',\iota) \to \VV^i_r(X,\iota)$ for all $i\le 1$ 
and $r\ge 0$, and induces isomorphisms between
$\VV^0_r(X',\iota)$ and $\VV^0_r(X,\iota) \cap 
f_{\sharp}^{*}\Hom (\pi_1(X'), G)$, for all $r\ge 0$.
\end{lemma} 

Finally, suppose $K(\pi,1)$ is a classifying space for the group 
$\pi$. In this case, we will simply denote the corresponding characteristic 
varieties by $\VV^i_r(\pi,\iota)$. If $X$ is a pointed space, and 
$f\colon X\to K(\pi,1)$ is a classifying map for its fundamental 
group, then the induced isomorphism  
$f_{\sharp}^{*}\colon  \Hom(\pi,G) \to \Hom(\pi,G)$ 
restricts to isomorphisms $\VV^1_r(\pi,\iota) \cong \VV^1_r(X,\iota)$ 
for all $r\ge 0$, see \cite[Cor.~5.11]{PS-16}.

\subsection{Flat connections}
\label{subsec:flat}

We now turn to the infinitesimal counterparts of the above constructions, 
following closely the exposition from \cite[\S\S2--3]{MPPS}. 
Let $(A,d)$ be commutative, differential graded algebra (for short, a $\dga$) 
over $\C$, and let $\g$ be a Lie algebra, also over $\C$.  The tensor product 
$A\otimes \g$ has the structure of a graded, differential Lie algebra, with 
Lie bracket given by
$[\alpha\otimes x, \beta\otimes y]= \alpha \beta \otimes [x,y]$, 
and differential given by $\partial (\alpha\otimes x) = d \alpha \otimes x$.
Clearly, this construction is functorial in both $A$ and $\g$.

The algebraic analogue of the $G$-representation variety $\Hom(\pi,G)$
is the (bi-functorial) pointed set $(\F(A,\g), 0)$  of {\em $\g$-valued flat connections}\/ 
on $A$, consisting of of degree $1$ elements in $A\otimes \g$ that
satisfy the Maurer--Cartan equation,
\begin{equation}
\label{eq:flat}
\partial\omega + \tfrac{1}{2} [\omega,\omega] = 0 .
\end{equation}

The $\dga$ $A$ is said to be connected if $A^0$ is the $\C$-span of $1$. 
For such an algebra, the bilinear map $P\colon A^1\times \g\to A^1\otimes \g$,
$(\eta,g)\mapsto \eta\otimes g$ induces a map
$P\colon H^1(A)\times \g \to \F(A,\g)$. 
The {\em essentially rank one}\/ part of $\F(A,\g)$ is the set
\begin{equation}
\label{eq:f1}
\F^1(A,\g):=P(H^1(A)\times \g).
\end{equation}

Suppose now that both $A^1$ and $\g$ are finite-dimensional. 
Then the set $\F(A,\g)$ has a natural structure of affine variety, which 
we shall call the {\em $\g$-variety of flat connections}\/ on $A$. 
Moreover, $\F^1(A,\g)$ is an irreducible, Zariski-closed subset 
of $\F(A,\g)$.  More precisely, $\F^1(A,\g)$ is either $\{0\}$, or the cone on
$\PP(H^1(A)) \times \PP(\g)$. 

An alternate interpretation of these varieties is given in \cite[\S4]{MPPS}.  
Set $A_i=(A^i)^*$, and let $\L(A_1)$ be the free Lie algebra on the dual 
vector space $A_1$.  We then define the {\em holonomy Lie algebra}\/ 
of $A$ as 
\begin{equation}
\label{eq:holo}
\h(A) := \L(A_1) / \ideal(\im(d^*+\cup^*)),
\end{equation}
where $d^*\colon A_2\to A_1=\L^1(A_1)$
and $\cup^*\colon A_2\to A_1\wedge A_1=\L^2(A_1)$ are the 
maps dual to the differential and the multiplication map in $A$, 
respectively.   Clearly, this construction is functorial.
Moreover, as shown in \cite[Prop.~4.5]{MPPS}, the canonical isomorphism
$A^1\otimes \g \cong \Hom (A_1,\g)$ restricts to a natural isomorphism 
\begin{equation}
\label{eq:holoflat}
\F(A,\g) \cong  \Hom_{\Lie} (\h(A), \g)
\end{equation} 
which identifies $\F^1(A,\g)$ with the set of Lie algebra morphisms from 
$\h(A)$ to $\g$ whose image is a vector subspace of dimension at most $1$.

Finally, let $\theta\colon \g\to \gl(V)$ 
be a finite-dimensional representation, and consider the set
\begin{equation}
\label{eq:piatheta}
\Pi(A,\theta)=P(H^1(A)\times \Zero(\det\circ \theta)),
\end{equation}
where $\det\colon \gl(V)\to \C$ is the determinant, and 
$\Zero(f):=\{x\mid f(x)=0\}$ is the zero set of a polynomial function $f$. 
Then $\Pi(A,\theta)$ is a Zariski-closed subset of $\F^1(A,\g)$ 
containing $0$.
Both $\F^1(A,\g)$ and $\Pi(A,\theta)$ behave functorially with respect to
$\dga$ maps. Moreover, $\dga$ maps inducing an $H^1$-isomorphism 
also induce $\F^1$ and $\Pi$-isomorphisms, since the variety $\F^1(A,\g)$ 
depends only on $H^1(A)$ and $\g$, and similarly $\Pi(A,\theta)$ depends 
only on $H^1(A)$ and $\theta$.

\subsection{Resonance varieties}
\label{subsec:res}

Once again, consider a  representation $\theta\colon \g\to \gl(V)$. 
For each flat connection $\omega\in \F(A,\g)$, we turn the 
tensor product $A\otimes V$ into a cochain complex,
\begin{equation}
\label{eq:aomoto}
\xymatrixcolsep{22pt}
\xymatrix{(A\otimes V , d_{\omega})\colon  \
A^0 \otimes V \ar^(.65){d_{\omega}}[r] & A^1\otimes V
\ar^(.5){d_{\omega}}[r]
& A^2\otimes V   \ar^(.55){d_{\omega}}[r]& \cdots },
\end{equation}
using as differential the covariant derivative
$d_{\omega}=d\otimes \id_V + \ad_{\omega}$.  Here, if 
we write $\omega=\sum_k a_k \otimes g_k$, for some 
$a_k\in A^1$ and $g_k\in \g$, then
$\ad_{\omega}(a\otimes v) = 
\sum_{k} a_k  a \otimes \theta(g_k)(v)$, 
for all $a \in A$ and $v\in V$.   It is readily checked that
the flatness condition on $\omega$ insures that $d_{\omega}^2=0$.

\begin{definition}
\label{eq:resvar}
The {\em resonance varieties}\/ of the $\dga$ $A^{\hdot}$ in 
degree $i\ge 0$ and depth $r\ge 0$ with respect to a representation 
$\theta\colon \g\to \gl(V)$ are the sets 
\begin{equation}
\label{eq:rra}
\RR^i_r(A, \theta)= \{\omega \in \F(A,\g)
\mid \dim H^i(A \otimes V, d_{\omega}) \ge  r\}.
\end{equation}
\end{definition}

For instance, $\RR^0_1(A,\theta)$ consists of those 
flat connections $\omega=\sum_k a_k \otimes g_k$ 
for which there exists a non-zero vector 
$v\in V$ such that $\theta(g_k)(v)=0$, for all $k$, 
see \cite[Lem.~2.3]{MPPS}, provided $A$ is  connected. 
For each $i\ge 0$, the sequence $\{\RR^i_r(A,\theta)\}_{r\ge 0}$ is 
a descending filtration of the set $\F(A,\g) = \RR^i_0(A,\theta)$. 
 We will refer to the pairs 
\begin{equation}
\label{eq:flatrv}
\left( \F(A,\g) , \RR^i_r(A, \theta) \right)
\end{equation}
as the (global) {\em infinitesimal embedded jump loci}\/ of $A$ with respect 
to $\theta$.  In the rank one case, i.e., the case when 
$\theta$ is the canonical identification $\C\to \gl_1(\C)$, we will simply 
write $\RR^i_r(A)$ for the corresponding sets.   If $H^{\hdot}(A)$ 
is the cohomology algebra of $A$, we will view it as a 
$\dga$ with differential $d=0$, and will denote the corresponding 
jump loci as $\RR^i_r(H^{\hdot}(A), \theta)$, or simply 
$\RR^i_r(H^{\hdot}(A))$ in the rank $1$ case.

Now suppose $A$ is a connected, finite-dimensional $\dga$, and the map 
$\theta\colon \g\to \gl(V)$ is a finite-dimensional representation of a 
finite-dimensional Lie algebra.  Then the sets $\RR^i_r(A, \theta)$ 
are closed subvarieties of $\F(A,\g)$, often referred 
to as the {\em resonance varieties}\/ of $A$ with respect to~$\theta$. 

They enjoy the following useful naturality property, proved in greater 
generality in \cite[Cor.~5.10]{PS-16}.

\begin{lemma}
\label{lem:rnat}
In the above setup, suppose $\psi\colon A'\to A$ is a morphism between two 
such $\dga$s, which is injective in degree $1$.   Then the natural morphism
\begin{equation}
\label{natinf}
\xymatrixcolsep{18pt}
\xymatrix{
\psi^{*} = \psi\otimes \id\colon  \F(A' , \g) \ar[r]&  \F(A , \g)
}
\end{equation}
is a closed embedding, which restricts to embeddings 
$\RR^i_r(A',\theta) \to \RR^i_r(A,\theta)$ for all $i\le 1$ 
and $r\ge 0$, and induces isomorphisms between $\RR^0_r(A',\theta)$
and $\RR^0_r(A,\theta) \cap \psi^{*}  \F(A' , \g)$, for all $r\ge 0$.
\end{lemma}

\subsection{Algebraic models and germs of jump loci}
\label{subsec:models}

Given a topological space $X$, we let $\Omega^{\hdot}(X)$ be the 
Sullivan algebra \cite{Su} of piecewise polynomial $\C$-forms on $X$. 
This $\dga$ has the property that $H^{\hdot}(\Omega(X))\cong H^{\hdot}(X,\C)$, 
as graded rings.  
A $\dga$ $A$ is said to be a {\em model}\/ for $X$ if $A$ may be connected 
by a zig-zag of quasi-isomorphisms to $\Omega(X)$.  For instance, if $M$ is 
a smooth manifold, then the de Rham algebra $\Omega_{\text{dR}} (M)$ of 
smooth $\C$-forms on $M$ is a model for this manifold.  We say that $A$ 
is a {\em finite}\/ model for $X$ if the dimension of $A$ (viewed as a $\C$-vector 
space) is finite, and $A$ is connected.

Assume now that $X$ is a path-connected space having the homotopy type 
of a finite CW-complex.  Let $\pi=\pi_1(X)$, let $\iota\colon G\to \GL(V)$ be 
a rational  representation, and let $\theta \colon \g \to \gl(V)$ be its 
tangential representation.  We will use frequently the following result, 
proved in \cite[Thm.~B(1)]{DP-ccm}. 

\begin{theorem}
\label{thm:b}
Suppose that $X$ admits a finite $\dga$ model $A$. 
There is then an analytic isomorphism of germs, 
$\F(A,\g)_{(0)} \isom \Hom(\pi,G)_{(1)}$, restricting to isomorphisms 
$\RR^i_r(A,\theta)_{(0)} \isom \VV^i_r(X,\iota)_{(1)}$, for all $i,r$. 
\end{theorem}

For later use, we will need the following lemmas. 

\begin{lemma}
\label{lem:abf}
Let $\pi=\Z^n$, with $n\ge 1$. Denote by $A_0$ the $\dga$
$(\bigwedge^{\hdot}H^1(\pi),d=0)$. Then $A_0$ is a finite model for 
the torus $T^n=K(\pi, 1)$, and
\[
(\F(A_0,\g),\RR^1_1(A_0,\theta)) = (\F^1(A_0,\g),\Pi(A_0,\theta)), 
\] 
for every  finite-dimensional Lie representation 
$\theta \colon \g\to \gl(V)$ of $\g=\sl_2$ or $\sol_2$. 
\end{lemma}

\begin{proof}
Since $T^n$ is a formal  space in the sense of Sullivan \cite{Su},
the $\dga$ $A_0$ is a finite model of $T^n$.  On the other hand, 
$A_0$ is the Chevalley--Eilenberg cochain $\dga$ of the (abelian) 
Malcev Lie algebra of $\pi= \Z^n$. Lemma 4.14 and Theorem 4.15 
from \cite{MPPS} together imply our second claim.
\end{proof}

\begin{lemma}
\label{lem:unip}
Let $\iota\colon G\to \GL(V)$ be a rational representation of a unipotent group.
If $\theta= d_1(\iota) \colon\g \to \gl(V)$, then $\F^1(A,\g)=\Pi(A,\theta)$, 
for any connected $\dga$ $A$. 
\end{lemma}

\begin{proof}
Since the group $G$ is unipotent, the homomorphism $\iota$ takes values 
in the upper triangular unipotent subgroup of $\GL_m$, where $m=\dim V$, 
by a classical result in representation theory \cite{Hu75}. Hence, the 
function $\det \circ \theta\colon \g\to \C$ is identically $0$.   
The claim then follows from the construction of $\F^1$ and $\Pi$.
\end{proof}

In the case when $G=\C$, more can be said. 
\begin{lemma}
\label{lem:unip-bis}
$\Hom (\Z^n, \C)_{(1)} = \VV^1_1(\Z^n, \iota)_{(1)}$, for all $n\ge 1$.
\end{lemma}

\begin{proof}
Consider the $\dga$ $A_0$ from Lemma \ref{lem:abf}. We infer from 
Theorem \ref{thm:b} that 
\begin{equation}
\label{eq:homzn}
(\Hom (\Z^n, \C), \VV^1_1(\Z^n, \iota))_{(1)} \cong 
(\F(A_0,\C),\RR^1_1(A_0,\theta))_{(0)}.
\end{equation}
Clearly, $\F(A_0,\C)=\F^1(A_0,\C)$. On the other hand, 
$\F^1(A_0,\C)=\Pi(A_0,\theta)$, by the preceding lemma. 
Finally, $\Pi(A_0,\theta) \subseteq \RR^1_1(A_0,\theta)$, 
by \cite[Thm.~1.2]{MPPS}.  
Therefore, $\F(A_0,\C)=\RR^1_1(A_0,\theta)$.
Our claim then follows from \eqref{eq:homzn}.
\end{proof}

\section{Quasi-K\"{a}hler manifolds and admissible maps}
\label{sect:qkahler}

\subsection{Admissible maps to curves}
\label{subsec:admissible}

Let $M$ be a connected, complex manifold.  We say that $M$ is 
a {\em quasi-K\"{a}hler manifold}\/  if $M= \oM \setminus D$, 
where $\oM$ is a connected, compact K\"{a}hler manifold and $D$ 
is a normal crossing divisor.  
A map $f\colon M \to C$ from such a manifold to a 
smooth complex curve $C$ is said to be {\em admissible}\/ if $f$ is 
holomorphic and surjective, and $f$ admits a holomorphic, surjective 
extension between suitable compactifications, $\bar{f}\colon \oM \to \oC$, 
such that all the fibers of $\bar{f}$ are connected. It is readily checked 
that the homomorphism on fundamental groups induced by such a 
map, $f_{\sharp}\colon \pi_1(M)\to \pi_1(C)$,  is surjective. 

We denote by $\cE(M)$ the family of admissible 
maps $f\colon M \to M_f$ to curves with negative 
Euler characteristic, modulo automorphisms of the target, 
and we denote by $f_{\sharp}\colon \pi\surj \pi_f$ the corresponding 
induced homomorphisms.  
Deep work of Arapura \cite{Ar} characterizes   
those irreducible components of the rank one characteristic 
variety $\VV^1_1(M)$ which contain the origin of the character 
group $\Hom(\pi, \C^{\times})$: all such components are connected, 
affine subtori, which can be described in terms of admissible maps, 
as follows.  

\begin{theorem}[\cite{Ar}] 
\label{thm:arapura}
For a quasi-K\"{a}hler manifold $M$, the set $\cE(M)$ is finite. 
Moreover, the correspondence 
$f \leadsto f_{\sharp}^{*} \Hom (\pi_f, \C^{\times})$
establishes a bijection between $\cE(M)$ and the 
set of positive-dimensional, irreducible components of 
$\VV^1_1(M)$ passing through $1$.
\end{theorem}

Let $\abf \colon \pi \surj \pi_{\abf}$ be the projection of the group 
$\pi$ onto its maximal torsion-free abelian quotient. 
We will denote by $f_0 \colon M \to K(\pi_{\abf}, 1)$ 
the corresponding classifying map, which is determined 
up to homotopy by the property that $(f_0)_{\sharp}= \abf$.   
Furthermore, we will write 
\begin{equation}
\label{eq:emab}
E(M)=\cE(M)\cup \{ f_0\}.
\end{equation}

\subsection{Cohomology jump loci of quasi-K\"{a}hler manifolds}
\label{subsec:cjl-qk}

Now let $G$ be a complex linear algebraic group, and let 
$\iota\colon G\to \GL(V)$ be a rational representation.  
By Lemma \ref{lem:cvnat},  the natural inclusion
\begin{equation}
\label{eq:repincl}
\Hom(\pi,G) \supseteq \bigcup_{f\in E(M)} f_{\sharp}^{*} \Hom (\pi_f,G)
\end{equation}
induces  an inclusion 
\begin{equation}
\label{eq:vincl}
\VV^i_r(\pi,\iota) \supseteq \bigcup_{f\in E(M)} f_{\sharp}^{*} \VV^i_r (\pi_f,\iota)
\end{equation}
for $i=0$ and $1$, and for all $r\ge 0$.  In order to prove Theorem \ref{thm:main}
from the Introduction, we want to establish some criteria under which 
the inclusions \eqref{eq:vincl} become equalities near $1$, for $i=1$ and $r=0,1$. 
In the case when $G=\C^{\times}$, $\iota=\id_{\C^{\times}}$, and 
$i=r=1$, equality near $1$  always holds in \eqref{eq:vincl}, and in fact 
is equivalent to Arapura's Theorem \ref{thm:arapura}.  To attack the 
general case, we start with some preliminary observations. 

Suppose first that $b_1(M)=0$. Plainly, $1\not\in \VV^1_1(\pi,\iota)$, and therefore
$\VV^1_1(\pi,\iota)_{(1)}= \O$.   Hence, equality \eqref{eq:vincl} follows trivially.   
Moreover, the natural map $\Omega (K(1,1)) \to \Omega (K(\pi,1))$ 
is a $1$-equivalence; hence, $\pi$ has the same $1$-minimal model 
as the trivial group, cf. \cite{DGMS, Su}. It then follows from \cite[Thm.~A]{DP-ccm} 
that 
\begin{equation}
\label{eq:hom1}
\Hom(\pi,G)_{(1)} = \{1\}.
\end{equation}
Therefore, equality \eqref{eq:repincl} also holds trivially in this case.  
Thus, we may assume from now on that $b_1(M)>0$. 

In view of the discussion at the end of \S\ref{subsec:repcv}, 
we may replace in \eqref{eq:vincl} the group $\pi$ 
by the manifold $M$, and likewise $\pi_f$ by $M_f$.
Moreover, for $i=r=1$, the characteristic variety $\VV^1_1(M_f,\iota)$ 
may be replaced in \eqref{eq:vincl} by the representation variety $\Hom (\pi_f,G)$,  
when $f\in \mathcal{E}(M)$.
Indeed, for each $f\in \mathcal{E}(M)$, the manifold $M_f$ is a 
connected, $2$-dimensional CW-complex with $\chi(M_f)<0$. 
Thus, by the computation from Example \ref{ex:2cw}, 
we have that 
\begin{equation}
\label{eq:v1surf}
\VV^1_1(M_f,\iota)=\Hom(\pi_f,G).
\end{equation}

Finally, let $f\in E(M)$. By Lemma \ref{lem:cvnat}, the set $f_{\sharp}^{*} \Hom (\pi_f,G)$
is Zariski closed in $\Hom (\pi,G)$, and the set $f_{\sharp}^{*} \VV^1_1 (\pi_f,\iota)$ 
is Zariski closed in $\VV^1_1(\pi,\iota)$. Furthermore, the analytic germ 
$f_{\sharp}^{*} \Hom (\pi_f,G)_{(1)}$ is isomorphic to $\Hom (\pi_f,G)_{(1)}$, and similarly
$f_{\sharp}^{*} \VV^1_1 (\pi_f,\iota)_{(1)} \cong \VV^1_1 (\pi_f,\iota)_{(1)}$.

\begin{remark}
\label{rk=i0}
We also deduce from Lemma \ref{lem:cvnat} that equality near $1$ in \eqref{eq:repincl}
implies equality near $1$ in \eqref{eq:vincl}, for $i=0$ and all $r\ge 1$.
\end{remark}

\section{Quasi-projective manifolds and Orlik--Solomon models}
\label{sect:qp-os}

\subsection{Orlik--Solomon models}
\label{subsec:os}
We now restrict our attention to a class of quasi-K\"{a}hler 
manifolds of great importance in complex algebraic geometry.  
Recall that a space $M$ is said to be a quasi-projective variety if 
$M$ is a Zariski open subset of a projective variety.  
By resolution of singularities, a connected, smooth, complex 
quasi-projective variety $M$ can realized as $M= \oM \setminus D$, 
where $\oM$ is a connected, smooth, complex projective variety, 
and $D$ is a normal crossing divisor.   For short, we will say that 
$M$ is a {\em quasi-projective manifold}.

Let $\overline{M}$ and $\overline{M'}$ be two projective manifolds, and 
let $D\subset \overline{M}$ and $D'\subset \overline{M'}$ be two divisors.
A {\em regular morphism}\/ of pairs, $\bar{f}\colon (\overline{M},D) \to (\overline{M}',D')$, 
is a regular map $\bar{f}\colon \overline{M}\to \overline{M}'$ with the property that 
$\bar{f}^{-1}(D')\subseteq D$. Clearly, the restriction 
$f\colon \overline{M}\setminus D \to \overline{M}'\setminus D'$ is also a regular map. 
Conversely, any regular map between quasi-projective manifolds is induced by a
regular morphism between suitable compactifications with normal crossing divisors, 
see \cite{Mo}.  Consequently, a  map between two quasi-projective manifolds, 
$f\colon M\to M'$, is admissible (in the sense of \S\ref{subsec:admissible}) 
if and only if $M'$ is a smooth curve and $f$ is a regular surjection with 
connected generic fiber. 

We will consider a class of divisors broader than the normal crossing type, 
namely the {\em hypersurface arrangements}\/ investigated in \cite{Du}. 
Extending Morgan's Gysin models from \cite{Mo}, Dupont constructs in \cite{Du} 
a bigraded $\Q$-$\dga$, $\OS^{\hdot}_{\hdot}(\overline{M},D)$,  
associated to a  hypersurface arrangement $D$ in $\overline{M}$,
functorial with respect to regular morphisms of such pairs. He proves that 
$\OS^{\hdot}(\overline{M},D)$ is a finite model of the quasi-projective manifold 
$M=\overline{M}\setminus D$.  It is straightforward to extract from the results 
in \cite{Du} that $\OS^{\hdot}_{\hdot}(\overline{M},D)$ is a model with positive 
weights for $M$, in the sense from \cite{PS-15}. Moreover, there is an identification,
$H^{\hdot}(\OS (\overline{M},D))\equiv H^{\hdot}(\overline{M} \setminus D)$, 
natural with respect to regular morphisms of pairs. 

Given a quasi-projective manifold $M$, a compactification $\overline{M}$ obtained 
by adding a hypersurface arrangement $D$ is called a {\em convenient 
compactification}\/ if every element of $\cE (M)$ is represented by an admissible 
map $f\colon M\to M_f$ which is induced by a regular morphism of pairs,
$\bar{f}\colon (\overline{M},D) \to (\overline{M}_f,D_f)$, where $\overline{M}_f$ 
is the canonical compactification of the curve $M_f$, obtained by adding a finite 
set of points $D_f$. It is known that convenient compactifications always exist, 
see \cite{Mo}. 
Fixing such an object, we will use the following simplified notation: 
for each $f\in \cE(M)$, we denote the weight-preserving $\dga$ map 
$\OS (\bar{f})\colon \OS(\overline{M}_f,D_f) \to \OS(\overline{M},D)$ by 
$\Phi_f\colon A_f \to A$. 

\begin{remark}
\label{rem:nonnat}
Let $M$ be a  quasi-projective manifold with fundamental group $\pi$, and 
let $A=\OS(\overline{M},D)$ be an Orlik--Solomon model for 
$M$.  If $G$ is a linear algebraic group whose Lie algebra $\g$ 
is abelian, then, as shown in \cite[Thm.~B(2)]{DP-ccm}, there is an analytic 
isomorphism of germs, 
\begin{equation}
\label{eq:iso fag}
\xymatrix{\F(A,\g)_{(0)} \ar^(.45){\cong}[r]& \Hom(\pi,G)_{(1)}},
\end{equation}
which is {\em natural}\/ with respect to the action on flat connections 
of $\dga$ maps induced by regular morphisms of pairs, and the action 
on representation varieties of  induced homomorphisms 
on fundamental groups.  Furthermore, this isomorphism 
restricts to isomorphisms $\RR^i_r(A,\theta)_{(0)} \isom 
\VV^i_r(M,\iota)_{(1)}$, for all $i,r$. 

The naturality of the isomorphism \eqref{eq:iso  fag} for $G=\SL_2(\C)$ 
and $\Sol_2(\C)$ would simplify the proof of Theorem \ref{thm:main}. 
As explained in \cite[\S 7.5]{PS-16}, though, the argument from 
\cite[Thm.~B(2)]{DP-ccm} that establishes the naturality of the 
isomorphism \eqref{eq:iso  fag} breaks down in the non-abelian case.
This is the reason why we chose to prove Theorem \ref{thm:main} 
with the aid of Lemma \ref{lem:irrtrick}, instead.
\end{remark}

\subsection{Flat connections and infinitesimal jump loci}
\label{subsec:flat-qp}
We now proceed to describe infinitesimal analogs of the inclusions \eqref{eq:repincl} 
and \eqref{eq:vincl} for $i=r=1$. Let $\theta\colon \g\to \gl (V)$ be a finite-dimensional 
representation of a finite-dimensional Lie algebra. By naturality of the set of flat 
connections, we have an inclusion,
\begin{equation}
\label{eq:flatincl}
\F(A,\g) \supseteq \F^1(A,\g)\cup 
\bigcup_{f\in \mathcal{E}(M)}  \Phi_f^{*} \F(A_f,\g)\, ,
\end{equation}
where $\Phi_f^{*}$ denotes the map 
$\Phi_f \otimes \id_{\g} \colon A_f^1 \otimes \g \to A^1 \otimes \g$.
This inclusion is then the analog of \eqref{eq:repincl}.
For the analog of \eqref{eq:vincl}, we need some preparation. 

\begin{lemma}
\label{lem:p411}
For every map $f\in \cE(M)$, the following hold: 
\begin{enumerate}
\item \label{w1} $A^{\hdot}_f=A^{\le 2}_f$; 
\item \label{w2} $\chi (H^{\hdot}(A_f))<0$; 
\item \label{w3} $\Phi_f$ is injective.
\end{enumerate}
\end{lemma}

\begin{proof}
The first claim follows easily from the construction of the Orlik--Solomon 
model $A_f$ of $M_f$, while the second claim simply translates the fact 
that $\chi(M_f)<0$. 

For the last assertion, we recall from \cite{Du} that there is a regular morphism 
of pairs, $\bar{b}\colon (\widetilde{M}, \widetilde{D}) \to (\overline{M},D)$, 
constructed by iterated blow-up, where $\widetilde{D}$ is a normal crossing 
divisor. We deduce that $\OS (\bar{f}\circ \bar{b})$ coincides with the map 
between Gysin models constructed by Morgan. This latter map is injective, 
as shown in \cite[Ex.~5.3]{DP-ccm}, and so we are done.
\end{proof}

Recall now that $b_1(M)>0$, and thus $H^1(A)\ne 0$.
The infinitesimal analog of \eqref{eq:vincl} is the following inclusion, 
\begin{equation}
\label{eq:rincl}
\RR^1_1(A,\theta) \supseteq  \Pi(A,\theta)\cup 
\bigcup_{f\in \mathcal{E}(M)}  \Phi_f^{*} \F(A_f,\g).
\end{equation}

Let us verify that \eqref{eq:rincl} holds.  
To start with, the inclusion $\Pi(A,\theta)\subseteq \RR^1_1(A,\theta)$ is 
given by \cite[Cor.~3.8]{MPPS}.  Next, for every $f\in \cE(M)$, 
we have that $\F(A_f,\g)=\RR^1_1(A_f,\theta)$, 
by \cite[Prop.~2.4]{MPPS} and Lemma \ref{lem:p411}.  Finally, 
$\Phi_f^{*} \RR^1_1(A_f,\theta) \subseteq \RR^1_1(A,\theta)$, 
by Lemmas \ref{lem:rnat} and \ref{lem:p411}.

\subsection{Properties of the infinitesimal inclusions}
\label{subsec:rk2}

Before proceeding, let us make a couple of  simple remarks about 
the inclusions in displays \eqref{eq:flatincl} and \eqref{eq:rincl}.  
All terms appearing on the right-hand side are Zariski closed 
subsets of $\F(A,\g)$, respectively $\RR^1_1(A,\theta)$. Indeed, 
for $\F^1(A,\g)$ and $\Pi(A,\theta)$, this follows by construction, 
cf.~\cite[\S 1.5]{MPPS}; 
moreover, these two varieties depend only on $H^1(A)$ and $\theta$. 
On the other hand, for $\Phi_f^{*} \F(A_f,\g)$ the claim follows from 
Lemmas \ref{lem:rnat}  and \ref{lem:p411}.  Furthermore, the analytic 
germs  $\Phi_f^{*} \F(A_f,\g)_{(0)}$ and $\F(A_f,\g)_{(0)}$ are isomorphic. 

\begin{lemma}
\label{lem:loctogl}
In both inclusions \eqref{eq:flatincl} and \eqref{eq:rincl}, 
equality is equivalent to equality near $0$.
\end{lemma}

\begin{proof}
The positive-weight decomposition of $A^1$ gives rise to a 
positive-weight $\C^{\times}$-action on $A^1\otimes \g$, leaving 
both $\F(A,\g)$ and $\RR^1_1(A,\theta)$ invariant, as explained in 
\cite[\S 9.17]{DP-ccm}. It follows from Lemmas \ref{lem:p1} and \ref{lem:p2} that 
equality near $0$ implies global equality.
\end{proof}

For later use, let us note that the above $\C^{\times}$-action on $A^1\otimes \g$ 
also endows with positive weights the subvarieties $\F^1(A,\g)$, $\Pi(A,\theta)$, and 
$\Phi_f^{*} \F(A_f,\g)$, for all $f\in \cE(M)$.

In the rank one case, i.e., when $\theta=\id_{\C}$, we have that 
$\F(A,\C)=H^1(A)$, for every connected $\dga$ $A$. Thus, inclusion \eqref{eq:rincl}
becomes 
\begin{equation}
\label{eq:r11incl}
\RR^1_1(A) \supseteq \{ 0\} \cup \bigcup_{f\in \mathcal{E}(M)} \im H^1(\Phi_f)
\end{equation}
in this case.  Actually, more can be said about this. Theorem C from \cite{DP-ccm}
implies the following infinitesimal analog of the bijection from Theorem \ref{thm:arapura}: 
\begin{equation}
\label{eq:infara}
\RR^1_1(A)= \{ 0\} \cup \bigcup_{f\in \mathcal{E}(M)} \im H^1(\Phi_f).  
\end{equation}
Moreover, this is the irreducible decomposition of $\RR^1_1(A)$, 
where  $\{ 0\}$ is omitted when $\cE(M) \neq  \O$, as in \cite[(50)]{MPPS}.

\begin{corollary}
\label{cor:1to2}
Equality in \eqref{eq:flatincl} implies equality in \eqref{eq:rincl}.
\end{corollary}

\begin{proof}
By Lemma \ref{lem:p411} and formula \eqref{eq:infara}, 
we may apply \cite[Prop.~4.1]{BMPP} to the family 
$\{ \Phi_f \}_{f\in \cE(M)}$ to obtain the desired conclusion.
\end{proof}

\section{Irreducibility, dimension, redundancies}
\label{sect:idr}

In this section, $M$ will be a quasi-projective manifold with 
fundamental group $\pi=\pi_1(M)$, and $A=\OS^{\hdot}(\overline{M},D)$ 
will be an OS-model for $M$ associated to a fixed convenient compactification 
of $M$.  Let $G$ be a complex linear algebraic group, 
let $\iota\colon G\to \GL(V)$ be a rational  representation, and let 
$\theta \colon \g \to \gl(V)$ be its tangential representation.  Unless 
otherwise mentioned, we suppose in this section that $G$ is 
either $\SL_2$ or its standard Borel subgroup $\Sol_2$,
consisting of upper-triangular matrices with determinant $1$. 
To avoid trivialities, we will assume throughout that $b_1(M)>0$. 

\subsection{Dimension and irreducibility}
\label{subsec:dimirr}  

Our strategy is to compare the union of germs 
at $1$ from \eqref{eq:repincl-intro} with the union of germs at $0$ from 
\eqref{eq:flatincl-intro}, and similarly for \eqref{eq:vincl-intro} and
\eqref{eq:rincl-intro}, using Theorem \ref{thm:b} and Lemma \ref{lem:irrtrick}.
We start this approach by verifying the dimension and irreducibility 
assumptions from that lemma.  

\begin{lemma}
\label{lem:flatirr}
Let $f\colon M\to M_f$ be an admissible map, and let $\Phi_f \colon 
A_f\to A$ be the corresponding morphism of $\dga$ models. 
For any complex linear algebraic group $G$,  the germs 
$f_{\sharp}^{*} \Hom (\pi_f,G)_{(1)}$ and $\Phi_f^{*} \F(A_f,\g)_{(0)}$ 
are isomorphic to $\F(H^{\hdot}(M_f), \g)_{(0)}$.  Moreover, if 
$G=\SL_2$ or $\Sol_2$, then
\begin{enumerate}
\item  \label{gg2}
These  germs are irreducible; 
\item \label{gg3}
$\F(A_f,\g)_{(0)}\ne \F^1(A_f,\g)_{(0)}$.
\end{enumerate}
\end{lemma}

\begin{proof}

As noted before,  there are isomorphisms of analytic germs, 
\begin{equation}
\label{eq:sharpisos}
f_{\sharp}^{*} \Hom (\pi_f,G)_{(1)} \cong  \Hom (\pi_f,G)_{(1)} \quad\text{and}\quad
\Phi_f^{*} \F(A_f,\g)_{(0)}\cong \F(A_f,\g)_{(0)}.
\end{equation}
On the other hand, by Theorem \ref{thm:b}, we also 
have an isomorphism 
\begin{equation}
\label{eq:moreiso}
\F(A_f,\g)_{(0)} \cong \Hom (\pi_f,G)_{(1)}.
\end{equation} 

Observe that the curve $M_f$ is a formal space in the sense of Sullivan \cite{Su}.
Hence, the $\dga$ $H^{\hdot}(M_f)$, endowed with zero differential, 
is another finite model of $M_f$.  Again by Theorem \ref{thm:b}, we 
have an isomorphism $\F(A_f,\g)_{(0)} \cong \F(H^{\hdot}(M_f),\g)_{(0)}$,
and this proves the first claim.

For Parts \eqref{gg2} and \eqref{gg3}, we will use Lemma 7.3 from \cite{MPPS}, 
which says that, for a smooth complex curve $C$ with $\chi(C)<0$, the variety 
$\F(H^{\hdot}(C),\g)$ is irreducible and strictly contains $\F^1(H^{\hdot}(C),\g)$.

Next, we prove \eqref{gg2}. By the aforementioned result, the variety 
$\F(H^{\hdot}(M_f),\g)$ is irreducible. 
Since this variety is homogenous, and thus has positive weights, 
Lemma \ref{lem:p2} implies that the germ $\F(H^{\hdot}(M_f),\g)_{(0)}$ is 
also irreducible. 

Finally, we prove \eqref{gg3}.  
Since $H^1(A_f)\cong H^1(M_f)$, there is an isomorphism of germs, 
$\F^1(A_f,\g)_{(0)}\cong \F^1(H^{\hdot}(M_f),\g)_{(0)}$.
Suppose that $\F(A_f,\g)_{(0)}= \F^1(A_f,\g)_{(0)}$.  Then the germ at $0$ of 
the variety $\F(H^{\hdot}(M_f),\g)$ would be isomorphic to the germ at $0$ 
of the closed subvariety $\F^1(H^{\hdot}(M_f),\g)$.  Hence, Lemma \ref{lem:p3} 
would imply that $\F(H^{\hdot}(M_f),\g)=\F^1(H^{\hdot}(M_f),\g)$, in 
contradiction with \cite[Lem.~7.3]{MPPS}.  The proof is thus complete.
\end{proof}

The next lemma completes the verification of the dimension and 
irreducibility assumptions from Lemma \ref{lem:irrtrick}.

\begin{lemma}
\label{lem:moreirr}
With notation as above, the following hold for $G=\SL_2$ or $\Sol_2$.
\begin{enumerate}
\item \label{gr1}
The germ $\abf^{*} \Hom (\pi_{\abf},G)_{(1)}$ is isomorphic to 
$\F^1(A,\g)_{(0)}$.
\item \label{gr2}
The germ  $\abf^{*} \VV^1_1 (\pi_{\abf},\iota)_{(1)}$ is isomorphic 
to $ \Pi(A,\theta)_{(0)}$.
\item \label{gr3}
All the above germs are irreducible.
\end{enumerate}
\end{lemma}

\begin{proof}
As mentioned previously, we have an isomorphism 
\begin{equation}
\label{eq:abfhom}
(\abf^{*} \Hom (\pi_{\abf},G), \abf^{*} \VV^1_1 (\pi_{\abf},\iota))_{(1)} \cong 
(\Hom (\pi_{\abf},G), \VV^1_1 (\pi_{\abf},\iota))_{(1)}.
\end{equation}
Denote by $A_0$ the $\dga$ $(\bigwedge^{\hdot}H^1(M),d=0)$.  
We deduce from Theorem \ref{thm:b}
and Lemma \ref{lem:abf} that
\begin{equation}
\label{eq:homcong}
(\Hom (\pi_{\abf},G), \VV^1_1 (\pi_{\abf},\iota))_{(1)} \cong 
(\F(A_0,\g),\RR^1_1(A_0,\theta))_{(0)}.
\end{equation}

Again by Lemma \ref{lem:abf}, 
\begin{equation}
\label{eq:facong}
(\F(A_0,\g),\RR^1_1(A_0,\theta)) = (\F^1(A_0,\g),\Pi(A_0,\theta)).
\end{equation}
Since plainly $H^1(A_0)\cong H^1(A)$, we have that 
\begin{equation}
\label{eq:f1cong}
(\F^1(A_0,\g),\Pi(A_0,\theta))\cong (\F^1(A,\g),\Pi(A,\theta)).
\end{equation}
Putting things together verifies claims \eqref{gr1} and \eqref{gr2}.

We now prove claim \eqref{gr3}.
The irreducibility of $\F^1(A,\g)_{(0)}$ follows from \cite[Lem.~3.3]{MPPS} 
and Lemma \ref{lem:p2}. By the construction of $\Pi(A,\theta)$ 
(see \cite[(18)]{MPPS}), the irreducibility of the zero set 
$\Zero (\det \circ \theta)$ implies the irreducibility of $\Pi(A,\theta)_{(0)}$. 
When $\g=\sl_2$, the variety $\Zero (\det \circ \theta)$ is irreducible, 
by \cite[Lem.~3.9]{MPPS}. 

Finally, we let $\g$ be the $2$-dimensional solvable Lie algebra $\sol_2$, 
with $1$-di\-mensional abelianization.  Set $m=\dim V$. Since 
$\sol_2$ is solvable, $\theta$ takes values in the upper triangular
Lie subalgebra of $\gl_m(\C)$, by a classical result in Lie theory 
\cite{Hu72}. Composing $\theta$ with the projection onto the 
diagonal matrices, we obtain a Lie algebra map 
$\theta' \colon \sol_2 \to \C^m$, with components $\lambda_1,\dots,\lambda_m$, 
having the property that $\det \circ \theta=\prod_i \lambda_i$. Since $\theta'$ factors 
through the abelianization, we infer that $\det \circ \theta= c\cdot \lambda^m$, 
for some constant $c$ and some linear map $\lambda$ on $\sol_2$, 
which clearly implies the irreducibility of $\Zero (\det \circ \theta)$. 
This completes our proof.
\end{proof}

\subsection{Non-redundant cases}
\label{subsec:nonred}

Next, we have to analyze the unions of germs at the origin from 
\eqref{eq:repincl-intro}--\eqref{eq:rincl-intro} from the viewpoint 
of their redundancies. 

\begin{lemma}
\label{lem:nonred1}
For any two distinct maps $f,g \in \cE(M)$, the following hold 
for $G=\SL_2$ or $\Sol_2$:
\begin{enumerate}
\item  \label{nr1}
$f_{\sharp}^{*} \Hom (\pi_f,G)_{(1)} \not\subseteq g_{\sharp}^{*} 
\Hom (\pi_g,G)_{(1)}$;
\item \label{nr2}
$\Phi_f^{*} \F(A_f,\g)_{(0)}  \not\subseteq \Phi_g^{*} \F(A_g,\g)_{(0)}$.
\end{enumerate}
\end{lemma}

\begin{proof}
We assume first that $f_{\sharp}^{*} \Hom (\pi_f,G)_{(1)} 
\subseteq g_{\sharp}^{*} \Hom (\pi_g,G)_{(1)}$.
Intersecting this inclusion with $\Hom (\pi ,\C^{\times})$ 
where $\C^{\times}$ is the subtorus of diagonal matrices 
from $G$, we infer that 
$f_{\sharp}^{*} \Hom (\pi_f,\C^{\times})_{(1)} 
\subseteq g_{\sharp}^{*} \Hom (\pi_g,\C^{\times})_{(1)}$.
Since both $f_{\sharp}^{*} \Hom (\pi_f,\C^{\times})$ and 
$g_{\sharp}^{*} \Hom (\pi_g,\C^{\times})$
are connected tori, it follows from Lemma \ref{lem:p1} that 
$f_{\sharp}^{*} \Hom (\pi_f,\C^{\times}) \subseteq g_{\sharp}^{*} 
\Hom (\pi_g,\C^{\times})$, in 
contradiction with the bijection from Theorem \ref{thm:arapura}. 

Finally, suppose that $\Phi_f^{*} \F(A_f,\g)_{(0)}  \subseteq 
\Phi_g^{*} \F(A_g,\g)_{(0)}$. Let $\C \subseteq \g$
be the Lie algebra of the torus considered above. 
Intersecting this inclusion with $\F(A, \C)$, we infer as before that
$\Phi_f^{*} \F(A_f,\C)_{(0)} \subseteq \Phi_g^{*} \F(A_g,\C)_{(0)}$, 
which implies that $\im H^1(\Phi_f) \subseteq \im H^1(\Phi_g)$, 
in contradiction with \eqref{eq:infara}. 
\end{proof}

\begin{lemma}
\label{lem:nonred2}
For all $f \in \cE(M)$ and for all complex linear algebraic groups $G$, 
the following equality of germs holds:
\begin{equation}
\label{eq:cap}
f_{\sharp}^{*} \Hom (\pi_f,G)_{(1)} \cap \abf^{*} \Hom (\pi_{\abf},G)_{(1)}=
f_{\sharp}^{*} \abf^{*} \Hom ((\pi_f)_{\abf},G)_{(1)}.
\end{equation}
Furthermore, if $G=\SL_2$ or $\Sol_2$, then 
\begin{enumerate}
\item \label{qq1}
$\Phi_f^{*} \F(A_f,\g)_{(0)}  \not\subseteq \F^1(A,\g)_{(0)}$; 
\item \label{qq2}
$f_{\sharp}^{*} \Hom (\pi_f,G)_{(1)} \not\subseteq \abf^{*} 
\Hom (\pi_{\abf},G)_{(1)}$.
\end{enumerate}
\end{lemma}

\begin{proof}
We first establish equality \eqref{eq:cap}. 
As explained in \cite[Rem.~7.8]{PS-16}, the analytic germ 
$\abf^{*} \Hom (\pi_{\abf},G)_{(1)}$ coincides with the abelian 
part near $1$ of the representation variety $\Hom (\pi,G)$, 
and similarly for $\pi_f$.  It follows that we may replace the map 
$\abf$ by the abelianization map $\ab$ in \eqref{eq:cap}.

The inclusion "$\supseteq$" is an immediate consequence 
of naturality of abelianization.  Hence, it is enough to prove 
that any homomorphism $\rho \colon \pi_f \to G$ for which $\rho \circ f_{\sharp}$  
factors through abelianization has the same property. This in turn follows from the fact the
epimorphism $f_{\sharp}$ induces a surjection on derived subgroups, plus naturality. 
Our first claim is proved.

To prove \eqref{qq1}, suppose that $\Phi_f^{*} \F(A_f,\g)_{(0)}  
\subseteq \F^1(A,\g)_{(0)}$.  Then, by the injectivity of $\Phi_f$, 
we would have that $\F(A_f,\g)_{(0)}  \subseteq \F^1(A_f,\g)_{(0)}$, 
in contradiction with Lemma \ref{lem:flatirr}\eqref{gg3}. 

To prove \eqref{qq2}, suppose that 
$f_{\sharp}^{*} \Hom (\pi_f,G)_{(1)} \subseteq \abf^{*} \Hom (\pi_{\abf},G)_{(1)}$.  
From \eqref{eq:cap}, we deduce that
$\Hom (\pi_f,G)_{(1)} = \abf^{*} \Hom ((\pi_f)_{\abf},G)_{(1)}$,
by the surjectivity of $f_{\sharp}$.  We now consider the inclusion 
$\F^1(A_f,\g)_{(0)}  \subseteq \F(A_f,\g)_{(0)}$. We claim that our assumption leads
to the equality $\F^1(A_f,\g)_{(0)} = \F(A_f,\g)_{(0)}$, the same contradiction as before.

In view of Lemma \ref{lem:p3}, our claim follows from the existence of an 
isomorphism of germs, $\F^1(A_f,\g)_{(0)} \cong \F(A_f,\g)_{(0)}$. 
To construct such an isomorphism, we consider the $\dga$ 
$A_0=(\bwedge^{\hdot}H^1(M_f),d=0)$ from Lemma \ref{lem:abf}. 
Since $A_f$ is a finite model of $M_f$, we have an isomorphism 
\begin{equation}
\label{eq:fafpif}
\F(A_f,\g)_{(0)} \cong \Hom (\pi_f,G)_{(1)},
\end{equation} 
by Theorem \ref{thm:b}. Next, 
\begin{equation}
\label{eq:pifpif}
\Hom (\pi_f,G)_{(1)} = \abf^{*} \Hom ((\pi_f)_{\abf},G)_{(1)},
\end{equation}
by our assumption, and clearly 
\begin{equation}
\label{eq:abfpif}
\abf^{*} \Hom ((\pi_f)_{\abf},G)_{(1)} \cong \Hom ((\pi_f)_{\abf},G)_{(1)}.
\end{equation}

Again by Theorem \ref{thm:b}, 
\begin{equation}
\label{eq:hpfa}
\Hom ((\pi_f)_{\abf},G)_{(1)} \cong  \F(A_0,\g)_{(0)}.
\end{equation}
By  Lemma \ref{lem:abf}, 
\begin{equation}
\label{eq:ff1a}
\F(A_0,\g)_{(0)}= \F^1(A_0,\g)_{(0)}.
\end{equation}
Finally, 
\begin{equation}
\label{eq:ff11}
\F^1(A_0,\g)_{(0)} \cong \F^1(A_f,\g)_{(0)},
\end{equation} 
since $H^1(A_0)\cong H^1(A_f)$.
Putting things together completes our proof.
\end{proof}

\subsection{Redundant cases}
\label{subsec:red}

We now complete our analysis of the redundancies in the unions
\begin{align}
\label{eq:repincl-union}
\abf^{*} \Hom (\pi_{\abf},G)_{(1)} &\cup \bigcup_{f\in \cE(M)} 
f_{\sharp}^{*} \Hom (\pi_f,G)_{(1)}, 
\\
\label{eq:vincl-union}
\abf^{*} \VV^1_1 (\pi_{\abf},\iota)_{(1)} &\cup \bigcup_{f\in \cE(M)} 
f_{\sharp}^{*} \Hom (\pi_f,G)_{(1)},
\\
\label{eq:flatincl-union}
\F^1(A,\g)_{(0)} &\cup \bigcup_{f\in \mathcal{E}(M)}  \Phi_f^{*} \F(A_f,\g)_{(0)},
\\
\label{eq:rincl-union}
\Pi(A,\theta)_{(0)} &\cup \bigcup_{f\in \mathcal{E}(M)}  \Phi_f^{*} \F(A_f,\g)_{(0)}.
\end{align}

By the results from \S\ref{subsec:nonred}, we need to consider the following two cases.
If either \eqref{eq:repincl-union} or \eqref{eq:vincl-union} is redundant, then 
\begin{equation}
\label{eq:red1}
\abf^{*} \VV^1_1 (\pi_{\abf},\iota)_{(1)} \subseteq f_{\sharp}^{*} \Hom (\pi_f,G)_{(1)}, 
\ \text{for some $f\in \cE(M)$}.
\end{equation}
If either \eqref{eq:flatincl-union} or \eqref{eq:rincl-union} is redundant, then
\begin{equation}
\label{eq:red2} 
\Pi(A,\theta)_{(0)}  \subseteq \Phi_f^{*} \F(A_f,\g)_{(0)},
\ \text{for some $f\in \cE(M)$}.
\end{equation}

\begin{lemma}
\label{lem:red2}
If condition \eqref{eq:red2} holds, then $b_1(M_f)=b_1(M)$.
\end{lemma}

\begin{proof}
Let $\C \subseteq \g$ be the Lie algebra of the unipotent group $\C \subseteq G$ 
consisting of the matrices of $\Sol_2$ with $1$'s on the diagonal.
Denote by $\iota'\colon\C \to \GL(V)$ the 
restriction of $\iota$. Clearly, $d_1(\iota')=\theta'$, the restriction of
$\theta$ to the Lie algebra $\C$. Our assumption implies that 
$\Pi(A,\theta')_{(0)}  \subseteq \Phi_f^{*} \F(A_f,\g)_{(0)}$. 
We infer from Lemma \ref{lem:unip} that 
\begin{equation}
\label{eq:facfac}
\F(A,\C)_{(0)}=\F^1(A,\C)_{(0)}  \subseteq \Phi_f^{*} \F(A_f,\g)_{(0)}.
\end{equation}

Therefore,
$\F(A,\C)_{(0)} \subseteq \Phi_f^{*} \F(A_f,\C)_{(0)}$. In other words, 
$H^1(A) _{(0)} \subseteq \im H^1(\Phi_f)_{(0)}$. Hence, 
\begin{equation}
\label{eq:b1phi}
b_1(M)=\dim H^1(A) _{(0)} \le \dim  \im H^1(\Phi_f)_{(0)}.
\end{equation}

On the other hand, $H^1(\Phi_f)$ is identified with $H^1(f)$,
as recalled in \S\ref{subsec:os}, and $H^1(f)$ is injective, 
since $f_{\sharp}$ is surjective. In conclusion, $b_1(M) \le  b_1(M_f)$. 
Since clearly $b_1(M) \ge b_1(M_f)$, we are done.
\end{proof}

\begin{lemma}
\label{lem:red1}
If condition \eqref{eq:red1} holds, then $b_1(M_f)=b_1(M)$.
\end{lemma}

\begin{proof}
Define $\iota'$ and $\theta'$ as before. Note that 
$\VV^1_1 (\pi_{\abf},\iota') \subseteq \VV^1_1 (\pi_{\abf},\iota)$,
by construction. By Lemma \ref{lem:unip-bis}, 
$\VV^1_1 (\pi_{\abf},\iota')_{(1)}= \Hom (\pi_{\abf}, \C)_{(1)}$.
Thus, we infer from our assumption that 
\begin{equation}
\label{eq:abfpish}
\abf^{*} \Hom (\pi_{\abf}, \C)_{(1)} \subseteq f_{\sharp}^{*} \Hom (\pi_f,G)_{(1)}.
\end{equation}
Hence, 
\begin{equation}
\label{eq:abfshpif}
\abf^{*} \Hom (\pi_{\abf}, \C)_{(1)} \subseteq f_{\sharp}^{*} \Hom (\pi_f, \C)_{(1)}.
\end{equation}
Consequently, $\dim \Hom (\pi_{\abf}, \C)_{(1)} \le \dim  \Hom (\pi_f, \C)_{(1)}$, 
that is, $b_1(M) \le b_1(M_f)$.
Proceeding now as in the proof of Lemma \ref{lem:red2} 
completes the proof of this lemma.
\end{proof}

\begin{lemma}
\label{lem:tori}
Suppose $b_1(M) = b_1(M_f)$ for some $f\in \cE(M)$. Then:
\begin{enumerate}
\item \label{tor1}
$H^1(f)$ is an isomorphism.
\item \label{tor2}
$\cE(M) = \{ f\}$. 
\item \label{tor3}
$\F^1(A,\g)\subseteq \Phi_f^{*} \F(A_f,\g)$, for 
any finite-dimensional Lie algebra $\g$.
\item \label{tor4}
$\abf^{*} \Hom (\pi_{\abf}, G)_{(1)} \subseteq f_{\sharp}^{*} \Hom (\pi_f,G)_{(1)}$, 
for any linear algebraic group $G$.
\end{enumerate}
\end{lemma}

\begin{proof}
\eqref{tor1}
This claim is clear, since $H^1(f)$ is injective. 

\eqref{tor2}
Fix $g\in \cE(M)$. We start by noting that 
$g_{\sharp}^{*} \Hom (\pi_g, \C^{\times})$ is a connected affine subtorus 
of dimension $b_1(M_g)$ of the connected affine torus 
$\Hom (\pi_{\abf}, \C^{\times})$ of dimension $b_1(M)$, 
since $(\pi_g)_{\ab}=H_1(M_g, \Z)$ has no torsion. 
Thus, our assumption implies that 
$f_{\sharp}^{*} \Hom (\pi_f, \C^{\times})= \Hom (\pi_{\abf}, \C^{\times})$.
We infer that $g_{\sharp}^{*} \Hom (\pi_g, \C^{\times}) \subseteq 
f_{\sharp}^{*} \Hom (\pi_f, \C^{\times})$.  By Theorem \ref{thm:arapura}, 
we must have $g=f$, and we are done.

\eqref{tor3}
It is enough to note that $\Phi_f^{*} \F^1(A_f,\g)= \F^1(A,\g)$, 
since $H^1(\Phi_f)\equiv H^1(f)$ is an isomorphism.

\eqref{tor4}
Let $f_{\abf}\colon\pi_{\abf} \surj (\pi_f)_{\abf}$ be the 
(surjective) homomorphism induced by $f_{\sharp}$. 
Since $b_1(M) = b_1(M_f)$, we have that 
$\pi_{\abf}\cong (\pi_f)_{\abf}$. Hence, $f_{\abf}$
is an isomorphism, and consequently
\begin{equation}
\label{eq:fabfhom}
f_{\abf}^{*} \Hom ((\pi_f)_{\abf},G)_{(1)}= \Hom (\pi_{\abf},G)_{(1)}.
\end{equation}
This equality of germs implies that 
\begin{equation}
\label{eq:abfhomsharp}
\abf^{*} \Hom (\pi_{\abf}, G)_{(1)} = f_{\sharp}^{*}\circ \abf^{*} \Hom ((\pi_f)_{\abf},G)_{(1)}
\subseteq f_{\sharp}^{*} \Hom (\pi_f,G)_{(1)},
\end{equation}
as asserted. This completes the proof.
\end{proof}

\section{Proofs of the main results}
\label{sect:pf}

In this section, we provide proofs to Theorems \ref{thm:tran}--\ref{thm:loci} 
from the Introduction.  

\subsection{Transversality in the quasi-projective setting}
\label{subsec:trans}

Let $M$ be a quasi-pro\-jective manifold, and 
fix a  compactification $\overline{M}=M\cup D$, where 
$\overline{M}$ is a connected, smooth projective 
variety, and $D$ is a hypersurface arrangement in $\overline{M}$.

As before, let $(A^{\hdot}_{\hdot}\, , d)=\OS^{\hdot}_{\hdot}(\overline{M},D)$ 
be the corresponding  Orlik--Solomon model for $M$.  
By construction, the lower degree (called weight) is 
concentrated in the interval $[i,2i]$, in (upper) degree $i$.
The terminology comes from the fact that the induced lower grading in 
cohomology splits Deligne's weight filtration \cite{Du}.
Thus, an element $\Omega \in A^1\otimes \g$ has weight decomposition 
$\Omega =\Omega_1 +\Omega_2$, where $\Omega_j\in A^1_j\otimes \g$, 
for $j=1,2$. Note also that $H^1(A)=A^1 \cap \ker (d)$, since $A$ is connected.
When we write $\Omega \in H^1(A)\otimes \g$, we mean that $\partial \Omega=0$, 
where $\partial$ denotes $d\otimes \id_{\g}$. By construction, 
$\Omega_1 \in H^1(A)\otimes \g$.

\begin{lemma}
\label{lem:oscurve}
Let $(A , d)=\OS(\overline{M},D)$ 
be any Orlik--Solomon model.  If $\Omega \in \F(A,\g)$ 
and $\Omega_1=0$, then $\Omega \in H^1(A)\otimes \g$, for any 
finite-dimensional Lie algebra $\g$.
\end{lemma}

\begin{proof}
For an element $\Omega \in A^1\otimes \g$, let us examine the flatness 
equation, $\partial \Omega+\frac{1}{2} [\Omega, \Omega]=0$.
Let $\Omega =\Omega_1 +\Omega_2$ be the weight decomposition. 
We recall from \S\ref{subsec:os} that both the differential and the product 
of $A$ have degree zero with respect to the weights of $A$. Using this fact, 
the weight $2$ component of the flatness equation translates to the equality
$\partial \Omega_2+\frac{1}{2} [\Omega_1, \Omega_1]=0$, which proves our claim.
\end{proof}

Fix now a convenient compactification $\oM=M\cup D$.  
Then any admissible map $f\colon M\to M_f$ induces a morphism $\Phi_f \colon 
A_f\to A$ between the corresponding OS-models, and this, in turn, induces a 
morphism $\Phi_f^{*}\colon \F(A_f,\g) \to \F(A,\g)$ between the corresponding 
varieties of $\g$-flat connections. 
The next result proves Theorem \ref{thm:tran}, Part \eqref{fg1} from the Introduction.

\begin{theorem}
\label{thm:nonabtrans}
Let $M$ be a quasi-projective manifold, and let $\g$ be a finite-dimensional 
Lie algebra. For any distinct $f,g\in \cE(M)$, 
\[
\Phi_f^{*} \F(A_f,\g) \cap \Phi_g^{*} \F(A_g,\g)= \{ 0\}.
\]
\end{theorem}

\begin{proof}
We start by noting that
\begin{equation}
\label{eq:rk1tran}
\im H^1(\Phi_f) \cap \im H^1(\Phi_g) = \{ 0\}.
\end{equation}
Indeed, $\im H^1(\Phi_f)$ is naturally identified with $\im H^1(f)$, 
and similarly for $g$. On the other hand, Theorem \ref{thm:arapura} 
yields a natural identification of $\im H^1(f)$ with the tangent space 
$T_1( f_{\sharp}^{*} \Hom (\pi_f, \C^{\times}))$ to the corresponding 
irreducible component through $1$ of $\VV^1_1(M)$, and similarly for $g$. 
Finally, as shown in \cite[Thm.~C(2)]{DPS-duke}, 
\begin{equation}
\label{eq:t1f}
T_1( f_{\sharp}^{*} \Hom (\pi_f, \C^{\times}))\cap T_1( g_{\sharp}^{*} 
\Hom (\pi_g, \C^{\times})) = \{ 0\}.
\end{equation}

Suppose now that $\Phi_f^{*}(\Omega)=\Phi_g^{*}(\Omega')$,  for some 
$\Omega \in \F(A_f,\g)$ and $\Omega' \in \F(A_g,\g)$. Consider the weight 
decompositions, $\Omega =\Omega_1 +\Omega_2$ and 
$\Omega' =\Omega'_1 +\Omega'_2$. We infer that 
$\Phi_f^{*}(\Omega_1)=\Phi_g^{*}(\Omega'_1)$, since Orlik--Solomon 
$\dga$ maps preserve weight. As mentioned before, 
$\Omega_1 \in H^1(A_f)\otimes \g$ and $\Omega'_1 \in H^1(A_g)\otimes \g$. 
We infer then from \eqref{eq:rk1tran} that 
$\Phi_f^{*}(\Omega_1)=\Phi_g^{*}(\Omega'_1)=0$. 
Hence, $\Omega_1=\Omega'_1=0$, since 
$\Phi_f$ and $\Phi_g$ are injective, by Lemma \ref{lem:p411}.
Our assumption becomes then $\Phi_f^{*}(\Omega_2)=\Phi_g^{*}(\Omega'_2)$. 
On the other hand, we know from Lemma \ref{lem:oscurve}
that $\Omega_2 \in H^1(A_f)\otimes \g$ and $\Omega'_2 \in H^1(A_g)\otimes \g$. 
Therefore, $\Phi_f^{*}(\Omega_2)=\Phi_g^{*}(\Omega'_2)=0$, by the same 
argument as before. Our proof is complete.
\end{proof}

Let us point out that the transversality property from Theorem \ref{thm:nonabtrans} 
is a non-abelian generalization of the aforementioned rank $1$ result from 
\cite[Thm.~C(2)]{DPS-duke}. 

\subsection{Topological and infinitesimal jump loci}
\label{subsec:pfmain}

We now turn to the proof of Theorem \ref{thm:main} from the 
Introduction. As before, we shall work with a fixed convenient 
compactification $\overline{M}$ of a quasi-projective 
manifold $M$ (which we will assume satisfies $b_1(M)>0$), 
and we shall let $A$ denote the corresponding Orlik--Solomon 
model for $M$.  The bulk of the proof is contained in the next 
three lemmas.  

\begin{lemma}
\label{lem:main step1}
If the inclusion \eqref{eq:flatincl-intro} is an equality, then 
the inclusion \eqref{eq:repincl-intro} becomes an equality near $1$.
\end{lemma}

\begin{proof}
Recall from Theorem \ref{thm:b} that 
$\F(A,\g)_{(0)} \cong  \Hom (\pi ,G)_{(1)}$, as analytic germs. 
By assumption, the inclusion \eqref{eq:flatincl-intro} is an equality 
near $0$.  Suppose first that the union
\eqref{eq:repincl-union} has no redundancies.  
By Lemma \ref{lem:irrtrick} and the results from \S\ref{subsec:dimirr},  
the inclusion \eqref{eq:repincl-intro} is then an equality near $1$, 
thereby verifying our claim.   

Now suppose that the union \eqref{eq:repincl-union} is redundant.  
Then \eqref{eq:red1} also holds. 
Hence, by Lemmas \ref{lem:red1} and \ref{lem:tori}, we have that 
$\cE(M) = \{ f\}$ and $H^1(f)$ is an isomorphism.   
Furthermore, by Lemma \ref{lem:tori} again, our claim in this case 
reduces to proving the equality
\begin{equation}
\label{eq:repincl-red}
\Hom (\pi ,G)_{(1)}= f_{\sharp}^{*} \Hom (\pi_f,G)_{(1)}.
\end{equation}

Our hypothesis regarding \eqref{eq:flatincl-intro} gives the equality
\begin{equation}
\label{eq:flatincl-red1}
\F(A,\g)_{(0)} = \F^1(A,\g)_{(0)} \cup \Phi_f^{*} \F(A_f,\g)_{(0)}.
\end{equation}
Again by Lemmas \ref{lem:red1} and \ref{lem:tori}, 
equation \eqref{eq:flatincl-red1} becomes
\begin{equation}
\label{eq:flatincl-red}
\F(A,\g)_{(0)} = \Phi_f^{*} \F(A_f,\g)_{(0)}.
\end{equation}
As seen before, $\Hom (\pi ,G)_{(1)}\cong \F(A,\g)_{(0)}$.
Plainly, $\Phi_f^{*} \F(A_f,\g)_{(0)}\cong \F(A_f,\g)_{(0)}$. Furthermore,
$\F(A_f,\g)_{(0)}\cong \Hom (\pi_f ,G)_{(1)}$, again by Theorem \ref{thm:b}.
Finally, $\Hom (\pi_f ,G)_{(1)}\cong  f_{\sharp}^{*} \Hom (\pi_f,G)_{(1)}$.

In conclusion, the equality from \eqref{eq:flatincl-red} implies that 
$\Hom (\pi ,G)_{(1)}\cong f_{\sharp}^{*} \Hom (\pi_f,G)_{(1)}$.
Therefore, by Lemma \ref{lem:p3}, equality \eqref{eq:repincl-red} holds, 
and we are done.
\end{proof}

\begin{lemma}
\label{lem:main step2}
If the inclusion \eqref{eq:flatincl-intro} is an equality, then 
the inclusion \eqref{eq:vincl-intro} becomes an equality near $1$.
\end{lemma}

\begin{proof} 
We infer from our assumption that the inclusion \eqref{eq:rincl-intro} is 
an equality, by Corollary \ref{cor:1to2}.
Set $X= \RR^1_1(A,\theta)_{(0)} \cong \VV^1_1 (\pi,\iota)_{(1)}$, 
cf.~Theorem \ref{thm:b}.  If the inclusion \eqref{eq:rincl-intro} is 
an equality near $0$ and the union \eqref{eq:vincl-union} has 
no redundancies, then the inclusion \eqref{eq:vincl-intro} is an 
equality near $1$, as claimed, by Lemma \ref{lem:irrtrick} 
and the results from \S\ref{subsec:dimirr}. 

If the union \eqref{eq:vincl-union} is redundant, we may assume also 
that \eqref{eq:red1} holds. Hence, $\cE(M) = \{ f\}$ and $H^1(f)$ is an 
isomorphism, by Lemmas \ref{lem:red1} and \ref{lem:tori}. 
By Lemma \ref{lem:tori}, we are left with  proving the equality
\begin{equation}
\label{eq:vincl-red}
\VV^1_1 (\pi,\iota)_{(1)}= f_{\sharp}^{*} \Hom (\pi_f,G)_{(1)}.
\end{equation}

Since the inclusion \eqref{eq:rincl-intro} is a global equality, 
we deduce the local equality
\begin{equation}
\label{eq:rincl-red1}
\RR^1_1(A,\theta)_{(0)}= \Pi(A,\theta)_{(0)} \cup \Phi_f^{*} \F(A_f,\g)_{(0)}.
\end{equation}
Again by Lemmas \ref{lem:red1} and \ref{lem:tori}, 
equality \eqref{eq:rincl-red1} becomes
\begin{equation}
\label{eq:rincl-red}
\RR^1_1(A,\theta)_{(0)}= \Phi_f^{*} \F(A_f,\g)_{(0)}.
\end{equation}
As we mentioned before, $\VV^1_1 (\pi,\iota)_{(1)}\cong \RR^1_1(A,\theta)_{(0)}$. 
Next, we have that 
\[
\Phi_f^{*} \F(A_f,\g)_{(0)}\cong \F(A_f,\g)_{(0)}\cong \Hom (\pi_f ,G)_{(1)}\cong  
f_{\sharp}^{*} \Hom (\pi_f,G)_{(1)},
\]
as in the proof of Lemma \ref{lem:main step1}. 
Therefore,  \eqref{eq:rincl-red} implies that 
$\VV^1_1 (\pi,\iota)_{(1)}\cong f_{\sharp}^{*} \Hom (\pi_f,G)_{(1)}$.
Hence, equality \eqref{eq:vincl-red} holds, by Lemma \ref{lem:p3}, 
and this completes our proof.
\end{proof}

\begin{lemma}
\label{lem:main step3}
If the inclusion \eqref{eq:repincl-intro} is an equality near $1$, then
the inclusion \eqref{eq:flatincl-intro} is also an equality.
\end{lemma}

\begin{proof}
By Lemma \ref{lem:loctogl}, it is enough to show that the inclusion 
\eqref{eq:flatincl-intro} becomes an equality near $0$. 
Set $X=\F(A,\g)_{(0)} \cong  \Hom (\pi ,G)_{(1)}$, cf.~Theorem \ref{thm:b}.
If the inclusion \eqref{eq:repincl-intro} is an equality near $1$ and the union
\eqref{eq:flatincl-union} has no redundancies, then 
the inclusion \eqref{eq:flatincl-intro} is an equality near $0$, by
Lemma \ref{lem:irrtrick} and the results from \S\ref{subsec:dimirr}. 

If the union \eqref{eq:flatincl-union} is redundant, we may assume 
also that \eqref{eq:red2} holds.  Hence, $\cE(M) = \{ f\}$ and $H^1(f)$ 
is an isomorphism, by Lemmas \ref{lem:red2} and \ref{lem:tori}. 
By Lemma \ref{lem:tori}, our claim reduces to verifying the equality 
\begin{equation}
\label{eq:flatincl-redbis}
\F(A,\g)_{(0)} = \Phi_f^{*} \F(A_f,\g)_{(0)}.
\end{equation}
Our assumption related to \eqref{eq:repincl-intro} gives the equality 
\begin{equation}
\label{eq:repincl-red1}
\Hom (\pi ,G)_{(1)}= \abf^{*} \Hom (\pi_{\abf}, G)_{(1)} \cup 
f_{\sharp}^{*} \Hom (\pi_f,G)_{(1)}.
\end{equation}
Again by Lemmas \ref{lem:red2} and \ref{lem:tori}, 
formula \eqref{eq:repincl-red1} reduces to
\begin{equation}
\label{eq:repincl-redbis}
\Hom (\pi ,G)_{(1)}= f_{\sharp}^{*} \Hom (\pi_f,G)_{(1)}.
\end{equation}
As seen before, $\F(A,\g)_{(0)} \cong  \Hom (\pi ,G)_{(1)}$. Clearly,
$f_{\sharp}^{*} \Hom (\pi_f,G)_{(1)}\cong  \Hom (\pi_f,G)_{(1)}$. Next, 
$\Hom (\pi_f,G)_{(1)}\cong \F(A_f,\g)_{(0)}$, again by Theorem \ref{thm:b}. 
Finally, $\F(A_f,\g)_{(0)}\cong \Phi_f^{*} \F(A_f,\g)_{(0)}$.

In conclusion, \eqref{eq:repincl-redbis} implies that 
$\F(A,\g)_{(0)} \cong \Phi_f^{*} \F(A_f,\g)_{(0)}$.
Hence, by Lemma \ref{lem:p3}, equality \eqref{eq:flatincl-redbis} 
holds, and we are done.
\end{proof}

\begin{proof}[Proof of Theorem \ref{thm:main}]
The implication \eqref{m1} $\Rightarrow$ \eqref{m4} follows from 
Lemma \ref{lem:main step3} and Corollary \ref{cor:1to2}. 
Implication \eqref{m4} $\Rightarrow$ \eqref{m3} is clear. 
The implication \eqref{m3} $\Rightarrow$ \eqref{m2} follows 
from Lemmas \ref{lem:main step1} and \ref{lem:main step2}. 
Finally, the implication \eqref{m2} $\Rightarrow$ \eqref{m1} is 
obvious. 
\end{proof}

\subsection{Irreducible decompositions} 
\label{subsec:pfirr}

We are now in a position to prove Theorem \ref{thm:irrtran} from 
the Introduction, regarding the decomposition into irreducible components 
of germs of embedded jump loci of a quasi-projective manifold $M$ 
satisfying one of the equivalent properties from Theorem \ref{thm:main}.  

\begin{proof}[Proof of Theorem \ref{thm:irrtran}]
We start with Parts \eqref{t2} and \eqref{t3}. 
The equalities \eqref{eq:repincl-irr}--\eqref{eq:rincl-irr} follow from Theorem \ref{thm:main}.
We also know from Lemmas \ref{lem:flatirr} and \ref{lem:moreirr} that all subgerms appearing 
in these unions are irreducible.  If any one of these unions has redundancies, then, in view 
of the results from \S\ref{subsec:red}, either 
\eqref{eq:red1} or \eqref{eq:red2} holds.
In Part \eqref{t2}, this violates our assumption on first Betti 
numbers, by Lemmas \ref{lem:red2} and \ref{lem:red1}. 
In Part \eqref{t3}, we have to verify equalities 
\eqref{eq:repincl-spec}--\eqref{eq:rincl-spec}:  these follow 
at once from \eqref{eq:repincl-irr}--\eqref{eq:rincl-irr} and Lemma \ref{lem:tori}.

To prove Part \eqref{t1}, we  start by examining the irreducible decomposition 
\eqref{eq:flatincl-irr}. By Theorem \ref{thm:nonabtrans}, all components different 
from $\F^1(A, \g)_{(0)}$ intersect pairwise in a single point.
Next, we claim that the irreducible decomposition \eqref{eq:repincl-irr} has the 
following property: all components different from $\abf^{*} \Hom (\pi_{\abf}, G)_{(1)}$ 
have positive-dimensional intersection with $\abf^{*} \Hom (\pi_{\abf}, G)_{(1)}$. 
Indeed, such an intersection is isomorphic to $\Hom ((\pi_f)_{\abf}, G)_{(1)}$, by
Lemma \ref{lem:nonred2}. Since $(\pi_f)_{\abf} \cong \Z^n$ with $n\ge 1$, 
Lemma \ref{lem:abf} gives the isomorphism 
$\Hom ((\pi_f)_{\abf}, G)_{(1)} \cong \F^1 (A_0, \g)_{(0)}$. 
On the other hand, the homogeneous variety $\F^1 (A_0, \g)$ is isomorphic to the cone 
on the product of projective spaces $\PP^{n-1}  \times \PP (\g)$, 
which implies that its germ at $0$ is positive-dimensional. 

By Theorem \ref{thm:b}, the germs $\Hom (\pi, G)_{(1)}$ and $\F (A, \g)_{(0)}$ are isomorphic.
Clearly, in Part \eqref{t1} we may suppose that $\cE (M)$ has at least two elements. 
With this assumption, we infer from Part \eqref{t2} and the above discussion that the 
isomorphism identifies the components $\abf^{*} \Hom (\pi_{\abf}, G)_{(1)}$ and 
$\F^1(A, \g)_{(0)}$. Indeed, the isomorphism identifies the components of 
$\Hom (\pi, G)_{(1)}$ with those of $\F (A, \g)_{(0)}$, modulo a permutation 
$\tau$ of the index set $E(M)$. Assume that $\F^1(A, \g)_{(0)}$ is identified 
with $\Hom (\pi_{f'}, G)_{(1)}$, for some $f' \in \cE(M)$, and pick an element 
$g' \in \cE(M)$ different from $f'$. By the above property of the irreducible 
decomposition \eqref{eq:flatincl-irr}, $\abf^{*} \Hom (\pi_{\abf}, G)_{(1)}$
must intersect $\Hom (\pi_{g'}, G)_{(1)}$ in a single point. On the other hand, 
the above property of the irreducible decomposition \eqref{eq:repincl-irr} implies 
that this intersection is positive-dimensional. This contradiction proves that 
$\tau (f_0)=f_0$, as claimed.  Our assertion in Part \eqref{t1} follows then from 
Theorem \ref{thm:nonabtrans}. 
\end{proof}

\begin{remark}
\label{rk:known}
We point out that all irreducible components appearing in Theorem \ref{thm:irrtran}
are known, for any quasi-projective manifold $M$ with $b_1(M)>0$
and any rational representation of $G=\SL_2$ or $\Sol_2$.
Indeed, Lemmas \ref{lem:flatirr} and \ref{lem:moreirr} provide isomorphisms of germs, 
$f_{\sharp}^{*} \Hom (\pi_f,G)_{(1)} \cong \Phi_f^{*} \F(A_f,\g)_{(0)} 
\cong \F(H^{\hdot}(M_f), \g)_{(0)}$, for any $f\in \cE (M)$, as well as isomorphisms 
$\abf^{*} \Hom (\pi_{\abf},G)_{(1)} \cong \F^1(A,\g)_{(0)}$ and 
$\abf^{*} \VV^1_1 (\pi_{\abf},\iota)_{(1)} \cong \Pi(A,\theta)_{(0)}$.
Finally, the affine varieties $\F(H^{\hdot}(M_f), \g)$ and $\F^1(A,\g)$, $\Pi(A,\theta)$ are
described in \cite[Lemmas 7.3 and  3.3]{MPPS}.
\end{remark}

\begin{remark}
\label{rem:irrglobal}
The equivalent properties from Theorem \ref{thm:main} also imply global equalities 
in \eqref{eq:flatincl-irr} and \eqref{eq:rincl-irr}, by Lemmas \ref{lem:p1} and \ref{lem:p2}. 
When all Betti numbers $b_1(M_f)$ 
are different from $b_1(M)$, these equalities are in fact global irreducible decompositions, 
by Theorem \ref{thm:irrtran}\eqref{t2} and Lemma \ref{lem:p2}. When 
$b_1(M_f)=b_1(M)$ for some $f\in \cE(M)$, the local equalities \eqref{eq:flatincl-spec} 
and \eqref{eq:rincl-spec} are actually global equalities of irreducible varieties, 
by a similar argument.
\end{remark}

\begin{remark}
\label{rem:shimura}
Building on the seminal work of Corlette and 
Simpson \cite{CS}, Loray, Pereira, and Touzet establish  
in \cite[Cor.~B]{LPT} the following striking  result.  
Let $M$ be a quasi-projective manifold, and let 
$\rho\colon \pi_1(M)\to \SL_2(\C)$ be a representation which is not 
virtually abelian.  There is then an orbifold morphism, $f\colon M\to N$, such that 
the associated representation, $\tilde\rho \colon \pi_1(M)\to  \PSL_2(\C)$, factors  
through the induced homomorphism $f_{\sharp}\colon \pi_1(M)\to \pi_1^{\orb}(N)$, where $N$ 
is either a $1$-dimensional complex orbifold, or a polydisk Shimura modular orbifold. 

The equality from  Theorem \ref{thm:irrtran}, display \eqref{eq:repincl-irr}  
provides a simpler, more precise local classification: If the representation $\rho$ is 
sufficiently close to the origin, then either $\rho$ is abelian, or there 
is an admissible map $f\colon M\to C$ such that $\rho$ factors through 
the homomorphism $f_{\sharp}\colon \pi_1(M)\to \pi_1(C)$, where $C$ 
is a smooth curve with $\chi(C)<0$.
\end{remark}

\begin{example}
\label{ex:case2}
Let $M$ be the product $\Sigma_g \times N$, where $\Sigma_g$ 
is a projective curve of 
genus $g>1$ and $N$ is a projective manifold with $b_1(N)=0$. This simple example 
shows that the  case from Theorem \ref{thm:irrtran}\eqref{t3} really does occur.
Indeed, it is clear that the canonical projection, $f\colon M\to \Sigma_g=M_f$, 
gives an element $f\in \cE(M)$ with $b_1(M_f)=b_1(M)$.
\end{example}

\subsection{On the structure of rank $2$ jump loci}
\label{subsec:top}

The results we have obtained so far enable us to derive structural 
decompositions near $1$ of the non-abelian rank $2$ 
topological embedded jump loci in low degree, for several large 
classes of quasi-projective manifolds.  These structural decompositions 
are summarized in Theorem \ref{thm:loci} from the Introduction, 
which we now proceed to prove.  

\begin{proof}[Proof of Theorem \ref{thm:loci}]
Let $M$ be a quasi-projective manifold; as explained in \S\ref{subsec:cjl-qk}, 
we may suppose that $b_1(M)>0$. Fix a convenient 
compactification $\oM=M\cup D$, and let $(A^{\hdot},d)=\OS^{\hdot}(\oM,D)$ 
be the corresponding model for $M$. 

We need to verify that equalities \eqref{eq:repincl-irr} and \eqref{eq:vincl-irr} 
hold in the five cases from our list. 
By Theorem \ref{thm:main}, it is enough to check 
that the infinitesimal inclusion \eqref{eq:flatincl-intro} is an equality 
in each case, that is, we need to verify that 
\begin{equation}
\label{eq:flat-equal}
\F(A,\g) = \F^1(A,\g)\cup 
\bigcup_{f\in \mathcal{E}(M)}  \Phi_f^{*} \F(A_f,\g),
\end{equation}
for each of the corresponding Orlik--Solomon models. 

\eqref{ai1}  First suppose that $M$ is projective.  In this case, 
$\OS^{\hdot}(M,  \O)=(H^{\hdot}(M), d=0)$.  In 
particular, $M$ is a formal space and $\pi$ is a $1$-formal group, 
and similarly for each curve $M_f$. It follows from \cite[Cor.~7.2]{MPPS} 
that \eqref{eq:flat-equal} holds. 

\eqref{ai2}  Next, suppose that the Deligne weight filtration has the 
property that $W_1H^1(M)=0$.  In this case, equality \eqref{eq:flat-equal}  
holds by \cite[Thm.~4.2]{BMPP}. 

\eqref{ai3} Now suppose that $M$ is the partial configuration space 
of a projective curve associated to a finite simple graph.
Then the needed equality is established  in \cite[Thm.~1.3]{BMPP}. 

\eqref{ai4}  Next, suppose that $\RR^1_1(H^{\hdot}(M), d=0)= \{ 0\}$.
Then  \eqref{eq:flat-equal} holds by  \cite[Cor.~7.7]{MPPS}.

\eqref{ai5} Finally, suppose that $M=S \setminus \{ 0\}$, where $S$ is 
a quasi-homogeneous affine surface 
having a normal, isolated singularity at $0$. 
Equality \eqref{eq:flat-equal} is then proved in displays 
(35) and (36) from \cite[Thm.~9.6]{PS-15}.
\end{proof}

\subsection{Transversality for K\"{a}hler manifolds and hyperplane arrangements}
\label{subsec:trans-bis}

When $M$ is either a compact, connected K\"{a}hler manifold 
or the complement of a central complex hyperplane arrangement,
the local analytic equalities \eqref{eq:repincl-irr} and \eqref{eq:vincl-irr} 
were obtained in  \cite[Thm.~1.3]{PS-16}, by a completely 
different approach.  (In the compact K\"{a}hler case, a map 
$f\colon M\to M_f$ is admissible if it is a holomorphic surjection 
with connected fibers onto a compact Riemann surface;  the 
finite set $E(M)$ is defined as before.) 

The method used in \cite{PS-16} is based on the fact that the family 
of maps $E(M)$ has the uniform formality property (in the sense of 
Definitions 3.2 and 6.3 from \cite{PS-16}), in the above two cases: 
this is proved in \cite[Prop.~7.4]{PS-16} for compact K\"{a}hler manifolds, 
respectively in  \cite[Prop.~9.3]{PS-16} for the arrangement case. 
This means that, for all $f\in E(M)$, there 
are zig-zags of augmentation-preserving quasi-isomorphisms 
connecting the Sullivan algebras of $M$ and $M_f$ to the 
respective cohomology algebras, as well as augmented $\dga$ maps $\Phi_i$ 
making the following ladder commute, up to augmented homotopy of $\dga$ maps,
\begin{equation}
\label{eq:ziggyf}
\begin{gathered}
\xymatrix{
\Omega(M)  & A_1 \ar_(.45){\psi_0}[l]  \ar^{\psi_1}[r] & \cdots 
& A_{\ell-1}   \ar[l]\ar^{\psi_{\ell-1}}[r] & H^*(M) \, \phantom{.}
\\
\Omega(M_f) \ar^{\Omega(f)}[u] & A'_1  
\ar^{\Phi_1}[u] \ar_(.45){\psi'_0}[l]  \ar^{\psi'_1}[r] & \cdots 
& A'_{\ell-1}   \ar_{\Phi_{\ell-1}}[u] \ar[l]\ar^{\psi'_{\ell-1}}[r] & H^*(M_f) \, , \ar_{f^*}[u]
}
\end{gathered}
\end{equation}
with the property that the isomorphism induced by the top zig-zag 
on deformation functors (i.e., the appropriate moduli spaces of flat 
connections) is independent of $f$. 

Using this uniform formality property, we obtain the following topological 
analog of the transversality property from Theorem \ref{thm:nonabtrans}, which 
proves Theorem \ref{thm:tran}, Part \eqref{fg2} from the Introduction.

\begin{theorem}
\label{thm:toptrans}
Let $M$ be either a compact, connected K\"{a}hler manifold or the 
complement of a central complex hyperplane arrangement.
Let $G$ be a linear algebraic group. Then
\[
f_{\sharp}^{*} \Hom (\pi_f,G)_{(1)} \cap g_{\sharp}^{*} \Hom (\pi_g,G)_{(1)} = \{ 1\},
\]
for any two distinct maps $f, g\in \cE(M)$.
\end{theorem}

\begin{proof}
In the arrangement case, we may suppose by a standard slicing argument that the 
hyperplanes lie in $\C^3$, since our claim depends only on the fundamental group 
$\pi=\pi_1(M)$. In both cases, we may choose a basepoint in $M$, and assume that
all elements of $E(M)$ are represented by pointed maps.

For a map $f\in E(M)$, consider the $\dga$ map 
$H^{\hdot}(f)\colon (H^{\hdot}(M_f),d=0) \to (H^{\hdot}(M),d=0)$, 
denoted $\Phi_f \colon A_f \to A$.   We want to apply \cite[Theorem 6.4]{PS-16} 
to the finite families $\{ f\}$ and $\{ \Phi_f\}$, for $q=1$. 
Clearly, all spaces and all $\dga$s appearing in these families are finite objects. 
Since each $f_{\sharp}$ is an epimorphism, both $f$ and $\Phi_f$ are $0$-connected 
maps. Denote by $\Omega (f)$ the $\dga$ map induced by $f$ between
Sullivan de Rham algebras. As mentioned before, $\Omega (f) \simeq \Phi_f$ in 
the category of augmented $\dga$s, uniformly with respect to $f\in E(M)$. 

Theorem 6.4 from \cite{PS-16} provides then a local analytic isomorphism, 
$\Hom (\pi,G)_{(1)} \cong \F(A,\g)_{(0)}$, that identifies 
$f_{\sharp}^{*} \Hom (\pi_f,G)_{(1)}$ with $\Phi_f^{*} \F(A_f,\g)_{(0)}$, 
for all $f\in \cE(M)$. Thus, our assertion will follow from the global 
transversality property 
\begin{equation}
\label{eq:inftran-appl}
\Phi_f^{*} \F(A_f,\g) \cap \Phi_g^{*} \F(A_g,\g)= \{ 0\}.
\end{equation}

In our situation, a stronger transversality holds, namely
\begin{equation}
\label{eq:inftran-strong}
\Phi_f^{*} A^1_f \otimes \g \cap \Phi_g^{*} A^1_g \otimes \g= \{ 0\}.
\end{equation}
Indeed, property \eqref{eq:inftran-strong} becomes
\begin{equation}
\label{eq:inftran-formal}
\im H^1(f) \otimes \g \cap \im H^1(g) \otimes \g= \{ 0\},
\end{equation}
by the construction of $\Phi$. On the other hand,
\begin{equation}
\label{eq:inftran-rk1}
\im H^1(f) \cap \im H^1(g) = \{ 0\},
\end{equation}
by the argument from the proof of  Theorem \ref{thm:nonabtrans}, which 
also works for quasi-K\"{a}hler manifolds. Thus, equality \eqref{eq:inftran-formal} 
holds, and this completes our proof.
\end{proof}

It is proved in Theorem 4.2 from \cite{DPS-imrn} that all pairs of distinct 
irreducible components of $\VV^1_1(M)$ intersect in a finite set, for any 
quasi-projective manifold $M$. In light of the bijection from Theorem \ref{thm:arapura}, 
Theorem \ref{thm:toptrans} may be viewed as a non-abelian analog of this rank $1$ result. 

\section{Rank greater than $2$}
\label{sect:rank3}

In this section, we consider in more detail the case when $M$ is a punctured 
quasi-homogeneous, isolated surface singularity, as in Theorem \ref{thm:loci}, 
Part \eqref{ai5}. 
For the group $\pi=\pi_1(M)$, we will examine the natural inclusion
\begin{equation}
\label{eq:repincl-rk3}
\Hom(\pi,G)_{(1)} \supseteq  \abf^{*} \Hom (\pi_{\abf},G)_{(1)}  \cup 
\bigcup_{f\in \cE(M)} f_{\sharp}^{*} \Hom (\pi_f,G)_{(1)}, 
\end{equation}
where $G=\SL_n(\C)$ with $n\ge 3$. 
We begin by recalling from \cite[\S{9}]{PS-15} several relevant facts. 

Since $M$ is a quasi-homogeneous variety, there is a 
positive weight $\C^{\times}$-action on $M$ with finite isotropy groups. 
The orbit space $M/\C^{\times}$ is a smooth projective curve $\Sigma_g$, 
where $g=b_1(M)/2$. Thus, our standard assumption that $b_1(M)>0$ 
translates to $g>0$.  

It is readily seen that the canonical projection, $f\colon M \to M/\C^{\times}=M_f$, 
is an admissible map.  Furthermore, 
\begin{equation}
\label{eq:emo}
\cE(M)= \begin{cases}
\O & \text{if $g=1$,}\\
\{ f\} & \text{if $g>1$.}
\end{cases}
\end{equation}

Set $H^{\hdot}=(H^{\hdot}(\Sigma_g), d=0)$. 
Define a $\dga$ $(A,d)$ by $A^{\hdot}= H^{\hdot}\otimes \bigwedge (t)$, 
with $t$ of degree $1$, where $d=0$ on $H^{\hdot}$ and $dt=\omega$, where
$\omega \in H^2$ is the orientation class. Then $A$ (respectively $H$) is 
a finite model of $M$ (respectively $M_f$). 

\begin{theorem}
\label{thm:rank3}
Let $M=S \setminus \{ 0\}$, where $S$ is a quasi-homogeneous affine surface 
having a normal, isolated singularity at $0$.  If $b_1(M)>0$ and $G=\SL_n(\C)$ 
with $n\ge 3$, then inclusion \eqref{eq:repincl-rk3} is strict. 
\end{theorem}

\begin{proof}
Assuming the contrary, we infer for $g>1$ that 
\begin{equation}
\label{eq:repincl-g2}
\Hom(\pi,G)_{(1)} = f_{\sharp}^{*} \Hom (\pi_f,G)_{(1)}.
\end{equation}
Indeed, in this case equality in \eqref{eq:repincl-rk3} becomes
\begin{equation}
\label{eq:hpi1}
\Hom(\pi,G)_{(1)} = \abf^{*} \Hom (\pi_{\abf},G)_{(1)}  \cup f_{\sharp}^{*} \Hom (\pi_f,G)_{(1)},
\end{equation}
since $\cE(M)=\{ f\}$. On the other hand, 
\begin{equation}
\label{eq:hpi2}
\abf^{*} \Hom (\pi_{\abf},G)_{(1)} \subseteq f_{\sharp}^{*} \Hom (\pi_f,G)_{(1)},
\end{equation}
by Lemma \ref{lem:tori}. 

If $g=1$, equality in \eqref{eq:repincl-rk3} becomes
\begin{equation}
\label{eq:repincl-g1}
\Hom(\pi,G)_{(1)} = \abf^{*} \Hom (\pi_{\abf},G)_{(1)},
\end{equation}
since $\cE(M)= \O$.

We will show that both \eqref{eq:repincl-g2} and \eqref{eq:repincl-g1} 
lead to a contradiction. We denote by $\varphi\colon H\inj A$ the 
canonical $\dga$ inclusion. Note that both $H$ and $A$ are
$\dga$s with positive weights, preserved by the map $\varphi$; 
see \cite[Prop.~9.1]{PS-15}.

First, we claim that equality \eqref{eq:repincl-g2} implies that 
\begin{equation}
\label{eq:flatincl-equal}
\F(A, \g)= \varphi^{*} \F(H, \g).
\end{equation}
To verify this claim, let us note that, by Lemmas \ref{lem:p1}--\ref{lem:p2}, it is 
enough to construct a local analytic isomorphism,
\begin{equation}
\label{eq:flatincl-cong}
\F(A, \g)_{(0)} \cong  \varphi^{*} \F(H, \g)_{(0)}.
\end{equation}
In turn, such an isomorphism is obtained as follows. First, 
$\F(A, \g)_{(0)} \cong  \Hom(\pi,G)_{(1)}$, by Theorem \ref{thm:b}.
Next, $\Hom(\pi,G)_{(1)} = f_{\sharp}^{*} \Hom (\pi_f,G)_{(1)}$, by  \eqref{eq:repincl-g2}. 
On the other hand, we clearly have that 
$f_{\sharp}^{*} \Hom (\pi_f,G)_{(1)} \cong \Hom (\pi_f,G)_{(1)}$ and 
$\varphi^{*} \F(H, \g)_{(0)}\cong \F(H, \g)_{(0)}$. Finally, 
$\Hom (\pi_f,G)_{(1)}\cong \F(H, \g)_{(0)}$,
again by Theorem \ref{thm:b}.  Thus, our claim is established.

Now, we claim that \eqref{eq:repincl-g1} also implies equality \eqref{eq:flatincl-equal}.
By the previous argument, it is enough to construct the isomorphism \eqref{eq:flatincl-cong}.
As before, $\F(A, \g)_{(0)} \cong  \Hom(\pi,G)_{(1)}$. 
Next, $\Hom(\pi,G)_{(1)} = \abf^{*} \Hom (\pi_{\abf},G)_{(1)}$, by  \eqref{eq:repincl-g1}. 
Plainly, $\abf^{*} \Hom (\pi_{\abf},G)_{(1)}\cong \Hom (\pi_{\abf},G)_{(1)}$ and 
$\varphi^{*} \F(H, \g)_{(0)}\cong \F(H, \g)_{(0)}$. 
Set $A_0= (\bigwedge^{\hdot} H^1(M), d=0)$. 
We infer from Lemma \ref{lem:abf} and Theorem \ref{thm:b} that 
$\Hom (\pi_{\abf},G)_{(1)}\cong \F(A_0, \g)_{(0)}$. 
Finally, the $\dga$s $A_0$ and $H$ are isomorphic, since $b_1(M)=b_1(M_f)$ and $g=1$. 

In conclusion, equality in \eqref{eq:repincl-rk3} implies \eqref{eq:flatincl-equal}, 
in all cases.

It will be convenient to rephrase equality \eqref{eq:flatincl-equal} in terms of the 
holonomy Lie algebra  $\h(A)$ described in \S\ref{subsec:flat}.  In view of the 
isomorphism \eqref{eq:holoflat}, the equality \eqref{eq:flatincl-equal} holds  
if and only if the natural morphism,
\begin{equation}
\label{eq:flatincl-hol}
\xymatrixcolsep{20pt}
\xymatrix{\h(\varphi)^{*} \colon \Hom_{\Lie}(\h(H),\g) \ar[r] &\Hom_{\Lie}(\h(A),\g)}, 
\end{equation}
is surjective.

Let $a_1,b_1,\dots,a_g,b_g$ be the dual of a symplectic basis of $H^1$, 
and let $\L$ be the free Lie algebra with this generating set. 
Write $r=\sum_{i=1}^g [a_i,b_i]$.
It is straightforward to check that $\h(H)$ is the quotient of $\L$ 
by the ideal generated by $r$, while $\h(A)$ is the quotient of $\L$ 
by the ideal generated by $[a_i,r]$ and $[b_i,r]$, for $i=1,\dots, g$. 
Moreover, the Lie morphism $\h(\varphi)\colon \h(A) \to \h(H)$ 
is the identity on free generators.  To disprove surjectivity 
in \eqref{eq:flatincl-hol}, we have to construct a Lie algebra map  
$\rho\colon \h(A) \to \g$ which does not factor through $\h(H)$. 

To achieve this goal, we first need to recall from \cite{Hu72} a couple of classical 
facts from the structure theory of semisimple Lie algebras. The elements of the 
root system $R$ of $\g=\sl_n$ are $t_{ij}:=t_i-t_j$, $1\le i\ne j \le n$, 
where $t_i$ denotes the $i$th projection of the Cartan subalgebra consisting of 
the diagonal matrices in $\sl_n$. For each such root, the corresponding $1$-dimensional 
root space $\g_{ij}$ is of the form $\C \cdot X_{ij}$, for some $X_{ij}\in \g$. It is known 
that $[X_{ij},X_{kl}]=0$ if $0\ne t_{ij}+t_{kl} \not\in R$, and 
$[X_{ij},X_{kl}]=c\cdot X_{i'j'}$ for some $c\in \C^{\times}$ if $t_{ij}+t_{kl}=t_{i'j'} \in R$.

Assuming that $n\ge 3$, we may now define the morphism $\rho\colon \h(A) \to \g$ 
by sending  the free Lie generators 
to $\rho(a_1)=X_{12}$, $\rho(b_1)=X_{23}$, and $\rho(a_i)=\rho(b_i)=0$ for $1<i\le g$. 
By the above discussion, $\rho(r)=c\cdot X_{13}$, for some $c\in \C^{\times}$. Since clearly 
$0\ne t_{12}+t_{13} \not\in R$ and $0\ne t_{23}+t_{13} \not\in R$, we have that 
$\rho ([a_i,r])=\rho ([b_i,r])=0$, for $i=1,\dots, g$. Hence, $\rho \in \Hom_{\Lie}(\h(A),\g)$. 
Plainly, the map $\rho$ does not factor through $\h(H)$, since $\rho(r)\ne 0$.  
This completes the proof.
\end{proof}

\section{Depth greater than $1$}
\label{sect:dep}

Let $M$ be a quasi-K\"{a}hler manifold, and 
let $\iota\colon G\to \GL(V)$ be a rational representation 
of a $\C$-linear algebraic group $G$.  By Lemma \ref{lem:cvnat}, 
we have for each $r\ge 0$ an inclusion of affine varieties,
\begin{equation}
\label{eq:vincldeg1}
\VV^1_r(\pi,\iota) \supseteq \bigcup_{f\in E(M)} f_{\sharp}^{*} \VV^1_r (\pi_f,\iota).
\end{equation}

For each $f\in E(M)$, we may view the induced homomorphism in cohomology,  
$\Phi_f :=H^{\hdot}(f) \colon H^{\hdot}(M_f) \to H^{\hdot}(M)$,  
as a map of $\dga$s with zero differentials.  Let
$\theta := d_1(\iota)\colon \g \to \gl (V)$ be the tangential Lie 
algebra representation.  By Lemma \ref{lem:rnat}, 
for each $r\ge 0$ we have  an inclusion of affine varieties,
\begin{equation}
\label{eq:rincldeg1}
\RR^1_r(H^{\hdot}(M),\theta) \supseteq 
\bigcup_{f\in E(M)} \Phi_f^{*} \RR^1_r (H^{\hdot}(M_f),\theta).
\end{equation}

\begin{lemma}
\label{lem:karr}
If $M$ is either a connected, compact K\"{a}hler manifold or the complement of a 
central complex hyperplane arrangement,  then the inclusion \eqref{eq:vincldeg1} 
becomes an equality near $1$ if and only if the inclusion \eqref{eq:rincldeg1} 
is an equality.
\end{lemma}

\begin{proof} 
By the argument from the proof of Theorem \ref{thm:toptrans}, we may apply 
\cite[Thm.~6.4]{PS-16} to the families $\{ f\in E(M)\}$ 
and $\{ \Phi_f \mid f\in E(M)\}$, 
for $q=1$. We obtain in this manner a local analytic identification, 
$\VV^1_r(\pi,\iota)_{(1)} \cong \RR^1_r(H^{\hdot}(M),\theta)_{(0)}$, 
which induces similar identifications, 
$f_{\sharp}^{*} \VV^1_r (\pi_f,\iota)_{(1)} \cong 
\Phi_f^{*} \RR^1_r (H^{\hdot}(M_f),\theta)_{(0)}$, 
for all $f\in E(M)$. Hence, \eqref{eq:vincldeg1} becomes an equality 
near $1$ if and only if \eqref{eq:rincldeg1} becomes an equality near $0$. 

On the other hand, the 
$\dga$ $(H^{\hdot}(M),d=0)$ has positive weights, equal to the degrees. 
As explained in \cite[\S 9.17]{DP-ccm}, this endows the variety 
$\RR^1_r(H^{\hdot}(M),\theta)$ with positive weights.
Our claim follows then from Lemmas \ref{lem:p1} and \ref{lem:p2}. 
\end{proof}

We continue with a more detailed analysis of inclusion \eqref{eq:vincldeg1} near $1$, 
in the rank $2$ case, i.e., when $G=\SL_2(\C)$ or $\Sol_2(\C)$. In the context of 
Lemma \ref{lem:karr}, we know from \cite[Thm.~1.3]{PS-16} that in this case 
\eqref{eq:vincldeg1} holds as an equality near $1$ for any $\iota$, 
when $r=0$ or $1$. What about depth greater than $1$?

\begin{lemma}
\label{lem:rk2strict}
Let $\theta \colon \g \to \gl (V)$ be a finite-dimensional Lie 
algebra representation of $\g=\sl_2$ or $\sol_2$ having a 
non-zero vector $v\in V$ annihilated by $\g$. If $M$ is a 
quasi-K\"{a}hler manifold with the property that there is an 
$f\in \cE(M)$ with $b_1(M_f)<b_1(M)$, then there is $r>1$ 
such that inclusion \eqref{eq:rincldeg1} is strict.
\end{lemma}

\begin{proof}
By \cite[Lemma 7.3]{MPPS}, there is a flat connection 
$\omega \in \F(H^{\hdot}(M_f), \g)$ which is not in 
$\F^1(H^{\hdot}(M_f), \g)$.  Set $\Omega =  \Phi_f^{*} (\omega)$. 
Clearly, $\Omega \in \F(H^{\hdot}(M), \g) \setminus \F^1(H^{\hdot}(M), \g)$,
since $H^1(f)$ is injective. 

Next, recall from \S\ref{subsec:res} that, given a finite 
$\dga$ $A$, a finite-dimensional Lie algebra representation 
$\theta \colon \g \to \gl (V)$,  and a flat connection $\omega \in \F(A,\g)$, 
there is an associated Aomoto cochain complex, 
$(A^{\hdot}\otimes V, d_{\omega})$, with differential 
$d_{\omega}^i \colon A^i\otimes V \to A^{i+1}\otimes V$.  
This gives rise to Aomoto--Betti numbers, $b_i(\omega, \theta):= 
\dim H^i(A\otimes V, d_{\omega})$.  By definition, 
$\omega \in \RR^i_k (A, \theta)$ if and only if $b_i(\omega, \theta)\ge k$. 
In the above setup, write $s=b_1(\omega, \theta)$ and $r=b_1(\Omega, \theta)$. 
By  \cite[Prop.~2.4]{MPPS}, we have that $s\ge 1$. We claim that $r>s$.

To verify the claim, let us consider the natural cochain map from \cite{DP-ccm},
\begin{equation}
\label{eq:nataom}
H^{\hdot}(f) \otimes \id_V \colon(H^{\hdot}(M_f)\otimes V, d_{\omega}) \to 
(H^{\hdot}(M)\otimes V, d_{\Omega}).
\end{equation} 
Since $H^{0}(M_f)=H^{0}(M)=\C$ and $H^1(f)$ is injective, we infer 
that $\dim \im (d_{\omega}^0)= \dim \im (d_{\Omega}^0)$.  Thus, 
we need to show that $\dim \ker (d_{\omega}^1)< \dim \ker (d_{\Omega}^1)$. 
To this end, we use our assumption that $b_1(M_f)<b_1(M)$ 
and pick a class $\eta \in H^1(M) \setminus \im H^1(f)$. 
Our hypothesis on $\theta$ yields the subspace
\begin{equation}
\label{eq:h1fsub}
H^{1}(f) \otimes \id_V ( \ker (d_{\omega}^1)) \oplus \C \cdot \eta\otimes v 
\subseteq H^1(M)\otimes V.
\end{equation}

By \eqref{eq:nataom}, we have an inclusion 
$H^{1}(f) \otimes \id_V ( \ker (d_{\omega}^1)) 
\subseteq \ker (d_{\Omega}^1)$.  If $\eta\otimes v \in \ker (d_{\Omega}^1)$, our 
claim follows. On the other hand, the property that $d_{\Omega}^1  (\eta\otimes v)=0$ 
is an immediate consequence of the fact that $\g$ annihilates $v$, by the construction 
of $d_{\Omega}$ recalled in \S\ref{subsec:res}.

Finally, we will show that 
\begin{equation}
\label{eq:omegarr}
\Omega \in \RR^1_r(H^{\hdot}(M),\theta) \setminus 
\bigcup_{g\in E(M)} \Phi_g^{*} \RR^1_r (H^{\hdot}(M_g),\theta).
\end{equation}

First, let us suppose that $\Omega \in \Phi_0^{*} \RR^1_r (H^{\hdot}(M_0),\theta)$. 
We know from Lemma \ref{lem:abf} that $\F(H^{\hdot}(M_0),\g)= \F^1(H^{\hdot}(M_0),\g)$. 
This leads to $\Omega \in \F^1(H^{\hdot}(M),\g)$, a contradiction. Next, assume that 
$\Omega \in \Phi_g^{*} \RR^1_r (H^{\hdot}(M_g),\theta)$, for some $g\in \cE(M)$. 
Consequently, 
\begin{equation}
\label{eq:omega-bis}
\Omega \in \im H^1(g)\otimes \g \cap \im H^1(f)\otimes \g,
\end{equation}
by the construction of $\Omega$.   If $g\ne f$, then $\im H^1(g)\cap \im H^1(f)=0$, by  
\cite[Thm.~C(2)]{DPS-duke}. This leads to $\Omega =0$, again a contradiction. 

Hence, $\Omega = \Phi_f^{*} (\omega')$, for some $\omega'\in \RR^1_r (H^{\hdot}(M_f),\theta)$.
Since $H^1(f)$ is injective, we infer that $\omega=\omega'$. Therefore, 
$s=b_1(\omega, \theta)=b_1(\omega', \theta)\ge r$.  This last contradiction 
completes our proof.
\end{proof}

Putting together Lemmas \ref{lem:karr} and \ref{lem:rk2strict},  we obtain the 
following theorem.

\begin{theorem}
\label{thm:depth2}
Let $M$ be either a connected, compact K\"{a}hler manifold, or the complement 
of a central complex hyperplane arrangement.  Let $\iota\colon G\to \GL(V)$ be 
a rational representation of $G=\SL_2(\C)$ or $\Sol_2(\C)$, having a non-zero 
fixed vector $v\in V^G$.  If there exists an admissible map $f\colon M\to M_f$ 
with $b_1(M_f)<b_1(M)$, then there is an integer $r>1$ such that inclusion 
\eqref{eq:vincldeg1} is strict near $1$.
\end{theorem}

\begin{example}
\label{ex:arrk}
Suppose $M$ is the complement of a central arrangement in $\C^3$.  There are then 
two cases to consider.  If the lines of the associated projective arrangement in 
$\CP^2$ intersect only in double points, it is well-known that the group 
$\pi=\pi_1(M)$ is free abelian. Hence, 
$\abf^{*} \colon \VV^1_r (\pi_{\abf},\iota) \to \VV^1_r (\pi_,\iota)$ 
is an isomorphism, and the inclusion \eqref{eq:vincldeg1} becomes a global equality, 
for any $\iota$ and all $r\ge 0$.  

On the other hand, if the lines in $\CP^2$ have an intersection point 
of multiplicity $m\ge 3$, we claim that Theorem \ref{thm:depth2} 
applies to $M$. Indeed,  \cite[Lem.~3.14]{Fa97} implies that 
$\RR^1_1(H^{\hdot}(M))$ has an irreducible component of dimension $m-1$. 
By \cite[Thm.~C(3)]{DPS-duke}, this component is of the form $\im (H^1(f))$, for some
$f\in \cE(M)$. Finally, $b_1(M_f)<b_1(M)$, since otherwise clearly 
$\RR^1_1(H^{\hdot}(M))=H^1(M)$, in contradiction with \cite[Thm.~2.8]{Fa97}.
\end{example}

Compact examples for Theorem \ref{thm:depth2} are also easy to construct. 

\begin{example}
\label{ex:cpt}
Let $M$ be the product $\Sigma_g \times N$, where $\Sigma_g$ 
is a projective curve of genus $g>1$ and $N$ is a projective manifold 
with $b_1(N)>0$. Plainly, the canonical projection 
$f\colon M\to \Sigma_g=M_f$ gives an element $f\in \cE(M)$ 
with $b_1(M_f)<b_1(M)$. 
\end{example}

\begin{ack}
We are grateful to Alex Dimca for very useful discussions related to the 
material from \S\ref{sect:local}.  We are also grateful to the referee for 
helpful comments and suggestions.  Some of the initial work on this paper 
was done while the second author visited the Institute of Mathematics of 
the Romanian Academy in June, 2016.  He thanks IMAR for its hospitality, 
support, and excellent research atmosphere.
\end{ack}

\newcommand{\arxiv}[1]
{\texttt{\href{http://arxiv.org/abs/#1}{arxiv:#1}}}
\newcommand{\arx}[1]
{\texttt{\href{http://arxiv.org/abs/#1}{arXiv:}}
\texttt{\href{http://arxiv.org/abs/#1}{#1}}}
\newcommand{\doi}[1]
{\texttt{\href{http://dx.doi.org/#1}{doi:#1}}}
\renewcommand{\MR}[1]
{\href{http://www.ams.org/mathscinet-getitem?mr=#1}{MR#1}}

\end{document}